\documentclass[11pt,twoside, dvipsnames]{amsart}

\usepackage{mathrsfs}

\usepackage{float}

\usepackage[all]{xy}
\usepackage{amsmath}  
\usepackage{mathtools}
\usepackage{amsfonts}
\usepackage{amssymb}
\usepackage{amsthm}
\usepackage{enumitem}
\usepackage{hyperref}
\usepackage{fullpage}
\usepackage{graphicx}
\usepackage{url}
\usepackage{comment}

\usepackage{todonotes}
\usepackage{xcolor}

\setlist[enumerate,1]{label=(\roman*)}
\setlist[enumerate,2]{label=(\alph*)}
\setlist[enumerate,3]{label=(\Roman*)}
\setlist[enumerate,4]{label=(\Alph*)}

\theoremstyle{definition}
\newtheorem{defn}{Definition}[section]

\newtheorem{rmk}[defn]{Remark}

\theoremstyle{plain}
\newtheorem{thm}[defn]{Theorem}

\newtheorem{lem}[defn]{Lemma}
\newtheorem{prop}[defn]{Proposition}
\newtheorem{cor}[defn]{Corollary}

\def\C{\ensuremath{\mathbb{C}}}

\def\H{\ensuremath{\mathbb{H}}}

\def\P{\ensuremath{\mathbb{P}}}

\def\FF{\ensuremath{\mathcal F}}
\def\GG{\ensuremath{\mathcal G}}
\def\HH{\ensuremath{\mathcal H}}
\def\II{\ensuremath{\mathcal I}}

\def\MM{\ensuremath{\mathcal M}}
\def\NN{\ensuremath{\mathcal N}}
\def\OO{\ensuremath{\mathcal O}}

\def\WW{\ensuremath{\mathcal W}}

\def\ch{\mathop{\mathrm{ch}}\nolimits}

\def\dim{\mathop{\mathrm{dim}}\nolimits}

\def\Hom{\mathop{\mathrm{Hom}}\nolimits}

\def\Hom{\mathop{\mathrm{Hom}}\nolimits}

\def\Imm{\mathop{\mathrm{Im}}\nolimits}

\def\NS{\mathop{\mathrm{NS}}\nolimits}

\def\rk{\mathop{\mathrm{rk}}}

\def\Real{\mathop{\mathrm{Re}}\nolimits}

\def\Stab{\mathop{\mathrm{Stab}}\nolimits}

\def\into{\ensuremath{\hookrightarrow}}
\def\onto{\ensuremath{\twoheadrightarrow}}
\def\sub{\ensuremath{\subset}}

\usepackage[all]{xy}
\usepackage{amsmath}
\usepackage{amsfonts}
\usepackage{amssymb}
\usepackage{amsthm}
\usepackage{color}
\usepackage{enumerate}
\usepackage{hyperref}
\usepackage{fullpage}
\usepackage{graphicx}

\usepackage{amsmath}

\usepackage{url}
\theoremstyle{definition}

\usepackage{hyperref}
\hypersetup{
    colorlinks=true,
    linkcolor=green!50!black,
    filecolor=magenta,      
    urlcolor=green!50!black,
    citecolor=blue!63!black
}

\newcommand{\nocontentsline}[3]{}
\newcommand{\tocless}[2]{\bgroup\let\addcontentsline=\nocontentsline#1{#2}\egroup}

\usepackage{fancyhdr}

\usepackage[margin=1in,headsep=.2in]{geometry}

\usepackage{enumitem}

\usepackage{fullpage}

\usepackage{url}
\usepackage{hyperref}
\usepackage{graphicx}
\usepackage[numbers]{natbib}
\usepackage{amsmath}
\usepackage{mathtools}
\usepackage{xcolor}
\usepackage{fullpage}

\usepackage{fancyhdr}
\usepackage{blindtext}

\usepackage{a4wide}

\usepackage{array, booktabs, makecell}

\usepackage[english]{babel}
\usepackage[utf8x]{inputenc}
\usepackage{amsmath}
\usepackage{tikz}

\usepackage{pict2e}

\usepackage{tikz}
\usepackage{calc}

\usepackage{pict2e}

\usepackage[normalem]{ulem} 
\usepackage{soul}

\usepackage[normalem]{ulem}

\usepackage{amsmath}
\usepackage{tikz}
\usepackage{calc}
\usepackage{tikz-cd}

\usetikzlibrary{calc} 
\usepackage{tikz-cd}
\usepackage{amsmath}

 \usepackage{graphicx}
\graphicspath{ {images/} }

\usepackage[english]{babel}
\usepackage{graphicx}

\usepackage{graphicx}

\usepackage{animate}

\usepackage{pgf,tikz}\usepackage{mathrsfs}\usetikzlibrary{arrows}

\usepackage{geometry}

\usepackage{mathtools}

\begin{document}

\title[An interesting wall-crossing]{\texorpdfstring{ An interesting wall-crossing: Failure of the wall-crossing/MMP correspondence}{ An interesting wall-crossing:Failure of the wall-crossing/MMP correspondence}}

\author{Fatemeh Rezaee}

\address{Centre for Mathematical Sciences, University of Cambridge, Wilberforce Road, CB3 0WA, Cambridge, United Kingdom}

\email{fr414@cam.ac.uk}

\begin{abstract}
We introduce a novel wall-crossing phenomenon in the space of Bridgeland stability conditions: a wall in the stability space of canonical genus $4$ curves that induces non-Q-factorial singularities and hence, it cannot be detected as an operation in the Minimal Model Program of the corresponding moduli space, unlike the case for many surfaces. 
  More precisely, we give an example of a wall-crossing in $\mathrm{D}^{b}(\mathbb{P}^{3})$ such that the wall induces a small contraction of the moduli space of stable objects associated to one of the adjacent chambers, but  a divisorial contraction to the other. This significantly complicates the overall picture in this correspondence to applications of stability conditions to algebraic geometry. The full wall-crossing for canonical genus four curves and the geometry are considered in the published paper \cite{R2}; this article is devoted to describe a particularly interesting wall among the walls in \cite{R2} in full details to explain the novel phenomenon.
\end{abstract}

\maketitle

\setcounter{tocdepth}{1}
\setcounter{secnumdepth}{3}

     \tableofcontents

\section{Introduction} \label{Introduction}

There  are many examples of moduli spaces of sheaves on surfaces whose entire MMP can explain and be explained by wall-crossing: each wall-crossing induces a birational map (as a MMP step in the movable cone), and every birational model associated to a movable divisor
on the moduli space appears as a moduli space $\mathcal{M}_{\sigma}(v)$ of $\sigma$-stable objects for some fixed vector $v$. If we replace surfaces by $\mathbb{P}^{3}$, 
this picture breaks down; indeed, our main result gives the first example of a wall-crossing from a smooth moduli space to a space whose main component is not even $\mathbb{Q}$-factorial.  This wall-crossing behaves in a manner that is surprising from the birational geometry:

\begin{thm} [See Theorem \ref{finally!}] \label{main Theorem} 
 Fix $v=(1,0,-6,15)$. There is a wall-crossing with respect to Bridgeland stability conditions $\mathcal{M}_{\sigma_{-}}(v) \rightarrow \mathcal{M}_{\sigma_{+}}(v)$ with the following properties:
 \begin{itemize}
     \item $\mathcal{M}_{\sigma_{-}}(v)$ is a smooth and irreducible variety,
     \item $\mathcal{M}_{\sigma_{+}}(v)= \widetilde {\mathcal{M}_{\sigma_{-}}(v)} \cup \mathcal{M}'$, where $ \widetilde {\mathcal{M}_{\sigma_{-}}(v)}$ is birational to $\mathcal{M}_{\sigma_{-}}(v)$ and $\mathcal{M}'$ is a new irreducible component,
     \item There is a diagram 
     \begin{center} 
    
\begin{tikzcd}[column sep=
scriptsize
]
 \mathcal{M}_{\sigma_{-}}(v) \arrow[dr, "\text{small contraction ($\phi$)}"{align=left,left=2mm,font=\scriptsize}] {}
&  &\widetilde {\mathcal{M}_{\sigma_{-}}(v)} \arrow[dl, "\text{divisorial contraction ($\psi$)}"{align=right,right=2mm,font=\scriptsize}] {} \\

& \mathcal{M}_{\sigma_{0}}(v)
\end{tikzcd}
\end{center} 
where both $\phi$ and $\psi$ have relative Picard rank 1. In particular,  $ \widetilde {\mathcal{M}_{\sigma_{-}}(v)}$ is not $\mathbb{Q}$-factorial.
 \end{itemize}

\end{thm}

Here, $v=(1,0,-6,15)$ is the Chern character of the ideal sheaf of a (2,3)-complete intersection curve in $\mathbb{P}^{3}$, $\mathcal{M}_{\sigma}(v)$ denotes the moduli space of $\sigma$-semistable objects with Chern character $v$, and $\sigma_{0}$ is a stability condition on the wall separating $\sigma_{-}$ and $\sigma_{+}$.

A smooth non-hyperelliptic genus 4 curve $C$ embeds into $\mathbb{P}^{3}$ as a (2,3)-complete intersection curve. 
The question is how to compactify the space of such curves. Considering Bridgeland stability conditions on $\mathrm{D}^{b}(\mathbb{P}^{3})$ gives a good answer: depending on a choice of a stability condition $\sigma \in \Stab(\mathbb{P}^{3})$, we obtain $\mathcal{M}_\sigma(1, 0, -6, 15)$, the moduli space of $\sigma$-stable complexes $E$ with $\mathrm{ch}(E)=\mathrm{ch}(\mathcal{I}_{C})$.
Following a path along the space of stability conditions, we want to understand how $\mathcal{M}_\sigma(1, 0, -6, 15)$ changes. Moving towards the large volume limit, at the beginning of the path, we get a very efficient compactification, given by a $\mathbb{P}^{15}$-bundle over $\P^9$, parametrising some non-torsion free sheaves in addition to the ideal sheaves. The second wall-crossing is what we describe in this paper. The full geometry of canonical genus four curves is described in the sequel \cite{R2}.

Let $P$ be a plane in $\mathbb{P}^{3}$,   $Z_{2} \subset P$ be a zero dimensional subscheme of length $2$,  $L$ a line in $\mathbb{P}^{3}$, and $\iota_ {P}$ the inclusion map from $P$ to  $\mathbb{P}^{3}$. Let $\mathcal{W}=\langle \mathcal{I}_{L}(-1),\iota_{P_{*}}(\mathcal{I}_{Z_{2}})^{\vee}(-5)\rangle$ be the wall defined by strictly semistable objects with Jordan-H\"older factors $\mathcal{I}_{L}(-1)$ and $\iota_{P_{*}}(\mathcal{I}_{Z_{2}})^{\vee}(-5)$.

Let $\sigma_{0} \in \mathcal{W}$, and $\sigma_{-}$ and $\sigma_{+}$ be stability conditions on either sides of the wall such that $\mathcal{I}_{L}(-1)$ has bigger phase than $\iota_{P_{*}}(\mathcal{I}_{Z_{2}})^{\vee}(-5)$ with respect to $\sigma_{-}$. We will see that $\mathcal{M}_{\sigma_{-}}(v)$ is a blow-up of a $\mathbb{P}^{15}$-bundle over $|\mathcal{O}(2)|=\mathbb{P}^{9}$ at a smooth center. On the other hand, we will show that $\mathcal{M}_{\sigma_{+}}(v)$ is the union  $\mathcal{M}_{\sigma_{+}}(v)= \widetilde {\mathcal{M}_{\sigma_{-}}(v)} \cup \mathcal{M}'$,  where  $\widetilde {\mathcal{M}_{\sigma_{-}}(v)}$ is as in Theorem \ref{main Theorem}, and $\mathcal{M}'$ is  a $\mathbb{P}^{17}$-bundle over $G(2,4) \times$ (7-dimensional locus) which appears as a new component after crossing the wall (or becomes unstable at the wall).

\subsection{Bridgeland Wall-crossing  versus Mori wall-crossing}

Consider the following principles: 

(1) Every wall-crossing is birational (in the sense that it is a step in the Minimal Model Program), or inducing a Mori fibration.

(2) These birational transformations/Mori fibrations are induced by a continuous map from $\Stab(S)$ to the movable cone of $\mathcal{M}_{\sigma_{-}}(v)$. Its image includes every chamber of the movable cone, and thus every birational model isomorphic to it up to codimension 2 appears as a moduli space.

These principles hold for a number of  cases:  Moduli of sheaves on K3 surfaces  (\cite{BM}); we will briefly explain this case in Section  \ref{BMsection}. For the Hilbert scheme of points on $\mathbb{P}^{2}$, some results can be found in \cite{ABCH, CH14b, BMW14, LZ}. Arguments for  Enriques surface can be found in \cite{NY} and \cite{Beckmann}, for Hirzebruch and del Pezzo surfaces in \cite{BC13}, for smooth projective surfaces in \cite{BHLRSWZ16}, and for abelian surfaces in \cite{Y, YY14, MM13}. The relation between Bridgeland stability and MMP for general smooth projective surfaces is discussed in \cite{Toda14} and  \cite{Toda13}. For surfaces with rational curves of
negative self-intersection, in \cite{TX}, the authors investigate a relationship between the geometry of such surfaces and Bridgeland stability conditions on the derived category.

 In more detail in the ideal situation, these results state the following:
For a relevant surface $S$, there is a \textit{linearization map} from any chamber in $\Stab(S)$ to the  nef cone  of the moduli space for the chamber. Gluing these maps for all chambers, one can define a global linearization map from  $\Stab(S)$ to the movable cone which relates birational transformations of the moduli space to wall-crossing in the stability space (\cite{BM}).

In our situation, we can still give a partial interpretation like above as follows. Roughly speaking, we will see that there are varieties $\mathcal{N}'$ and $\mathcal{N}''$ birational to  $\mathcal{M}_{\sigma_{-}}(v)$, and a map from the stability space to the movable cone of those varieties, such that the image of the map is completely contained in the walls of the movable cone.

The first examples of wall-crossing for curves in $\mathbb{P}^{3}$ can be found in \cite{Sch, Xia} for twisted cubics, and in \cite{GHS} for elliptic quartics. Unlike the case for these curves, there is no classical result on canonical genus four curves, which is considered in the present paper and its sequel \cite{R2}.

The main goal of the paper is to investigate the birational behaviour of hitting the wall $\mathcal{W}$. In particular, we study the exceptional loci in  $\mathcal{M}_{\sigma_{-}}(v)$ and $\mathcal{M}_{\sigma_{+}}(v)$, and also the intersection of the two components in $\mathcal{M}_{\sigma_{+}}(v)$. Regarding this, we will prove the following Theorem:
\begin{thm} [See Theorem \ref{intersec}]  \label{intN3}
The intersection $\widetilde {\mathcal{M}_{\sigma_{-}}(v)} \cap \mathcal{M}'$ is the exceptional divisor of the contraction map $\psi$. This exceptional locus contains an open subset  such that the restriction of $\psi$ to it is a $\mathbb{P}^{13}$-bundle over a 10-dimensional base. Over a 7-dimensional subset of the base, the fibers degenerate to a 14-dimensional cone with $\mathbb{P}^9$ as vertex over the 5-dimensional variety of rank one $2 \times 4$ matrices.
\end{thm}
\noindent\textbf{Strategy of the proof.} To prove Theorem \ref{main Theorem} and make a birational description of our wall, we need Theorem \ref{intN3} which basically describes the exceptional locus of the birational map associated with the wall. We observe that the morphism $ \widetilde{\psi}\colon \mathcal{M}_{\sigma_+}(v) \rightarrow \mathcal{M}_{\sigma_0}(v)$ identifying S-equivalent objects with respect to stability conditions on the wall, contracts the new component $\mathcal{M}'$ as a $\mathbb{P}^{17}$-bundle. Therefore, in order to understand $\psi=\widetilde{\psi}|_{\widetilde{\mathcal{M}_{\sigma_-}(v)}}$, we first need to understand the intersection of each $\mathbb{P}^{17}$ with $\widetilde{\mathcal{M}_{\sigma_-}(v)}$. There are a couple of steps to reach this goal. First, after some heavy Ext computations, we determine the singular locus of $ \mathcal{M}_{\sigma_+}(v)$ along $\mathcal{M}'$ (which will eventually be the intersection of the two components).
Secondly, we will explicitly construct enough degenerations of objects in $\mathcal{M}' \cap \widetilde{\mathcal{M}_{\sigma_-}(v)}$ to objects in $\mathcal{M}'$  to recover the 14-dimensional cone. The moduli space $\MM'$ contains stable pairs whose underlying curve is the union of a plane quintic with a line intersecting this quintic, along with two marked points on the quintic. We show that such stable pairs arise as the degeneration of the ideal sheaf of (2,3)-complete intersection curves if and only if the quintic has two nodes that are colinear with the intersection point with the line, and if the two marked points are the nodes. In this case, the plane quintic arises as the projection of a (2,3)-complete intersection curve in $\P^3$ from the intersection point with the line (Lemma \ref{{H}^{i}(E)} and Lemma \ref{pointx}). The next step is to degenerate the union to a plane quartic union a thick line passing through those nodes, via Lemma \ref{closure}.

After all, we are able to analyse our wall-crossing in Section \ref{Birational morphism corresponding to the wall-crossing} and  prove our main Theorem (see Theorem \ref{finally!}). We explain the birational situation and the relation between wall-crossing in the stability manifold and in the movable cone of some intermediate spaces birational to our component.
 
$$$$
\textbf{Acknowledgements.} First and foremost, I would like to thank Arend Bayer for suggesting the problem, continued support, enormous generosity of time and invaluable comments on preliminary versions. Special thanks to Benjamin Schmidt for the great suggestion to work on the stable pairs side, and helpful discussions. This work benefited from useful discussions with Antony Maciocia and Diletta Martinelli. I am grateful for comments by Aaron Bertram, Ivan Cheltsov, Daniel Huybrechts, Dominic Joyce, Emanuele Macr\`{i},  Balázs Szendröi, Richard Thomas, Yukinobu Toda, Bingyu Xia and Ziquan Zhuang. The author was supported by PCDS scholarship, the school of mathematics of the university of Edinburgh scholarship, ERC Starting grant WallXBirGeom, no. 337039, and ERC Consolidator grant WallCrossAG, no. 819864. Part of the work was written when the author was visiting the mathematical institute of the university of Bonn, and she would like to thank them for their hospitality. This material is partially based upon work supported by the NSF under Grant No. DMS-1440140 while the author was in residence at the Mathematical Sciences Research Institute in Berkeley, California, during the spring 2019 semester. The author was also partially supported by EPSRC grant EP/T015896/1,  during final edits. The pictures were generated with GeoGebra, and some computations were checked using Macaulay2 (\cite{GS}).  The author was supported by UKRI grant No. EP/X032779/1 and the ERC Advanced Grant MSAG during the final stages of revision.

$$$$
\textbf{Notation.} 
\begin{center}
   \begin{tabular}{ r l }
     $\mathcal{M}_{\sigma_{+}}(v)$: & moduli space of $\sigma_{+}$-stable objects with respect to $v$\\
      $\mathcal{M}_{\sigma_{-}}(v)$: & moduli space of $\sigma_{-}$-stable objects with respect to $v$\\
       $\mathcal{M}'$: & the new component in $\mathcal{M}_{\sigma_{+}}(v)$ after crossing the wall $\mathcal{W}$\\
       $\widetilde {\mathcal{M}_{\sigma_{-}}(v)}$: & the component birational to $\mathcal{M}_{\sigma_{-}}(v)$ in $\mathcal{M}_{\sigma_{+}}(v)$\\
       $\langle A,B \rangle$: & the wall  describes strictly semistable objects with Jordan-H\"older factors $A$\\ & and  $B$\\
      $\Stab(\mathbb{P}^{3})$: & the space of Bridgeland stability conditions in $\mathbb{P}^{3}$  with  respect to\\& $v=(1,0,-6,15)$\\
      $\Stab^{tilt}(\mathbb{P}^{3})$: & the space of tilt-stability conditions in $\mathbb{P}^{3}$  with respect to $v=(1,0,-6,15)$\\
       $\lambda_{\alpha, \beta, s}$: & Bridgeland stability slope\\
      $\nu_{\alpha, \beta}$: & tilt-stability slope\\
      $\H$: & (the left branch of) the hyperbola defined by 
      $\Imm(Z_{\alpha, \beta, s}(1,0,-6,15))=0$\\
      $N_k(X)$: & numerical group of k-cycles on a scheme $X$\\
      $\HH^{i}$: & the $i-th$ cohomology object in the corresponding heart\\
     $\mathrm{H}^{i}$: & the $i-th$ sheaf cohomology  group\\
      $\otimes$: & derived tensor (unless otherwise is explicitly stated)
   \end{tabular}
\end{center}
$$$$
\textbf{Conventions.} When there is no confusion, the subobject and the quotient of the defining short exact sequence of any wall will be denoted by $A$ and $B$, respectively. Notice that we cross the walls towards the large volume limit. 

Also, notice that as we will see, $\MM_1$ will be a projective bundle, and hence $\mathcal{M}_{\sigma_{-}}(v)$ as its blow up at a locus, will be a smooth reduced moduli space. We consider $\mathcal{M}_{\sigma_{+}}(v)$ as an algebraic space given by the reduced part of the moduli space defined by the union $ \widetilde {\mathcal{M}_{\sigma_{-}}(v)} \cup \mathcal{M}'$.

\section{Bridgeland stability  on $\mathbb{P}^{3}$ and Bayer-Macr\`{i} linearization map} \label{background}
In this section, we briefly review Bridgeland stability conditions on $\mathbb{P}^{3}$, and the Bayer-Macr\`{i} map for K3 surfaces in order to compare it with our case for $\mathbb{P}^{3}$.

\subsection{Bridgeland stability conditions on $\mathbb{P}^{3}$}

In this subsection, we briefly define stability conditions on the bounded derived category $\mathrm{D}^b(\mathbb{P}^{3})$, of coherent sheaves on $\mathbb{P}^{3}$, following the construction in ~\cite{BMT14:stability_threefolds}. Let $\mathrm{Coh}(\mathbb{P}^{3})$ be the abelian category of coherent sheaves (as an initial heart of a bounded t-structure with usual notion of torsion pair) on  $\mathbb{P}^{3}$. Let $\alpha>0, \beta \in \mathbb{R}$ be two real numbers. For $E \in \mathrm{D}^{b}(\mathbb{P}^{3})$, we define the \textit{twisted slope function} by $\mu_{\beta}(E):=  \frac{c_{1}(E)-\beta c_{0}(E)}{c_{0}(E)}$  if $c_{0}(E) \neq 0$, and $\mu_{\beta}(E)=+ \infty$ otherwise. We define the twisted Chern characters $\mathrm{ch}^{\beta}(E)= e^{-\beta H}.\mathrm{ch}(E)$, where $\mathrm{H}$ denotes the hyperplane class.

\begin{defn}
        By \textit{tilting}, one can define a new heart of a bounded t-structure as follows: the torsion pair is defined by 
       \begin{equation*}
\begin{aligned}\label{(2.1)}
\mathcal{T}_{\beta}:= \text{{\{$E \in \mathrm{Coh}(\mathbb{P}^{3})\colon \mu_{\beta}(G)>0$  for all $E  \twoheadrightarrow G $}\}},   \\
\mathcal{F}_{\beta}:= \text{{\{$E \in \mathrm{Coh}(\mathbb{P}^{3})\colon \mu_{\beta}(F) \leq 0$ for all $F  \hookrightarrow E$}\} }.
\end{aligned}
\end{equation*}
The new \textit{heart of a bounded t-structure} can be defined as $\mathrm{Coh}^{\beta}(\mathbb{P}^{3}):= \langle \mathcal{F}_{\beta}[1], \mathcal{T}_{\beta} \rangle$.

 The \textit{central charge} and the corresponding \textit{slope function} for the new heart can be defined as

\begin{equation*}
Z^{tilt}_{\alpha, \beta}:=-(\mathrm{ch}_{2}-\beta \mathrm{ch}_{1}+(\beta^{2}/2)\mathrm{ch}_{0})+(\alpha^{2}/2)\mathrm{ch}_{0}+i(\mathrm{ch}_{1}-\beta \mathrm{ch}_{0})=-(\mathrm{ch}^{\beta}_{2})+(\alpha^{2}/2)\mathrm{ch}^{\beta}_{0}+i(\mathrm{ch}^{\beta}_{1}), 
\end{equation*}
\hfill \break
and (using the twisted notation)
\begin{equation*}
\nu_{\alpha, \beta}:=-\frac{\mathrm{Re}(Z^{tilt}_{\alpha, \beta})}{\mathrm{Im}(Z^{tilt}_{\alpha, \beta})}=\frac{\mathrm{ch}^{\beta}_{2}-(\alpha^{2}/2)\mathrm{ch}^{\beta}_{0}}{\mathrm{ch}^{\beta}_{1}},
\end{equation*}

with $\nu_{\alpha, \beta}(E)=+\infty$ if $\mathrm{ch}^{\beta}_{1}(E)=0$. The pair $(\mathrm{Coh}^{\beta}(\mathbb{P}^{3}),Z^{tilt}_{\alpha, \beta})$ is called \textit{tilt-stability}.
\end{defn}

We denote by $\Stab^{tilt}(\mathbb{P}^{3})$, the space of all tilt-stability conditions. It was conjectured in ~\cite{BMT14:stability_threefolds} for arbitrary three-folds, and proved in by Macr\`{i} in \cite{M} for $\mathbb{P}^{3}$ that tilting again gives a Bridgeland stability condition. We will define this briefly.

\begin{defn}
We define a \textit{torsion pair} similarly as for the tilting case:

 \begin{equation*}
\begin{aligned}\label{(2.2)}
\mathcal{T}_{\alpha,\beta}:=\text{{\{$E \in \mathrm{Coh}^{\beta}(\mathbb{P}^{3}): \nu_{\alpha, \beta}(G)>0$ for all $E  \twoheadrightarrow G$ }\}},   \\
\mathcal{F}_{\alpha,\beta}:=\text{{\{$E \in \mathrm{Coh}^{\beta}(\mathbb{P}^{3}): \nu_{\alpha, \beta}(F) \leq 0$ for all $F  \hookrightarrow E$}\} }.
\end{aligned}
\end{equation*}

Now define a new \textit{heart, central charge, and slope} respectively as follows:

\begin{equation*}\mathrm{Coh}^{ \alpha, \beta}(\mathbb{P}^{3}) :=\langle \mathcal{F}_{\alpha,\beta}[1], \mathcal{T}_{\alpha,\beta}\rangle,
\end{equation*}

\begin{equation*}
Z_{\alpha,\beta,s}:=-\mathrm{ch}^{\beta}_{3}+(s+1/6)\alpha^{2}\mathrm{ch}^{\beta}_{1}+i(\mathrm{ch}^{\beta}_{2}-\alpha^{2}/2\mathrm{ch}^{\beta}_{0}),
\end{equation*}
and (for $s>0$)

\begin{equation*}
\lambda_{\alpha, \beta, s}:= -\frac{\mathrm{R}e(Z_{\alpha,\beta,s})}{\mathrm{I}m(Z_{\alpha,\beta,s})}.
\end{equation*}
with $\lambda_{\alpha, \beta, s}(E)=+\infty$ if $\mathrm{Im}(Z_{\alpha,\beta,s})(E)=0$.
\end{defn}

The pair $\sigma_{\alpha, \beta, s}=(\mathrm{Coh}^{\alpha, \beta}(\mathbb{P}^{3}), Z_{\alpha,\beta,s})$ (when exists) is called \textit{Bridgeland stability}.
   
   Before going further, we have a formal definition of a wall and chamber:
   
   \begin{defn} [{\cite {MS2}, \cite{MS3}}]
   
  A \textit{numerical wall} in Bridgeland stability with respect to
a class $v \in \Lambda$ is a non trivial proper subset of the stability space which is defined as

$$\mathcal{W}_{v,v'}=\{\sigma \in Stab (\mathbb{P}^3): \lambda_{\alpha, \beta, s}(v)=\lambda_{\alpha, \beta, s}(v'),  \text{  for any $v' \in \Lambda$}\}.$$
(We omit $v,v'$ in the notation, when it's clear from the text.)

 An \textit{actual wall} is a subset $\mathcal{W}'$ of a numerical wall so that the set of semistable objects with
class $v$ changes at $\mathcal{W}'$ (we can give a similar definition for tilt-stability).

A \textit{chamber} is defined as a connected component of the complement of the set of actual walls.

   \end{defn}
   
The main step to show that $(\mathrm{Coh}^{\alpha, \beta}(\mathbb{P}^{3}), Z_{\alpha,\beta,s})$ defines a Bridgeland stability condition (for all $s>0$) is a Bogomolov-type inequality, which we will refer to it as \textit{BMT inequality} (Theorem \ref{BMT14:stability_threefolds}). Before stating that, we have the classical Bogomolov-Gieseker inequality:
\begin{thm}  [{\cite[Corollary 7.3.2]{BMT14:stability_threefolds}}] \label{BG}
 Any $\nu_{\alpha, \beta}$-semistable object $E \in \mathrm{Coh}^{\beta}(\mathbb{P}^{3})$ satisfies
 $$2\mathrm{ch}_{0}(E)\mathrm{ch}_{2}(E) \leq \mathrm{ch}^{2}_{1}(E).$$
\end{thm}

\begin{thm}
 [{\cite[Lemma 8.8]{BMS}, \cite[Theorem 1.1]{M}}] \label{BMT14:stability_threefolds}
 
 Any $\nu_{\alpha, \beta}$-semistable object $E \in \mathrm{Coh}^{\beta}(\mathbb{P}^{3})$ satisfies
\begin{align*}
\alpha^{2} [(\mathrm{ch}^{\beta}_{1} (E))^{2}-2(\mathrm{ch}^{\beta}_{0} (E) (\mathrm{ch}^{\beta}_{2} (E) )]+4(\mathrm{ch}^{\beta}_{2} (E))^{2}-6(\mathrm{ch}^{\beta}_{1} (E)) \mathrm{ch}^{\beta}_{3} (E) \geq 0,\end{align*}
 and therefore  $(\mathrm{Coh}^{\alpha, \beta}(X), Z_{\alpha,\beta,s})$ is a Bridgeland stability condition for all $s \geq 0$. The support
property is also satisfied.
\end{thm}

The support property implies the manifold of all (Bridgeland) stability conditions $\Stab(\mathbb{P}^{3})$ admits a chamber decomposition, depending on $v$, such that
(i) for a chamber $C$, the moduli space $\mathcal{M}_{\sigma}(v) = \mathcal{M}_{C}(v) $ is independent of the choice of
$\sigma \in C$, and
(ii) walls consist of stability conditions with strictly semistable objects of class $v$ (\cite{BM}).

It turns out that there is a well-behaved wall-chamber structure in $\Stab^{tilt}(\mathbb{P}^{3})$ (as well as $Stab(\mathbb{P}^{3})$). The last part of the following Theorem was proved for surfaces in \cite{MacA}:
\begin{thm}[{\cite{BMS}}] \label{WC}

The function $\mathbb{R}_{> 0} \times \mathbb{R} \rightarrow \Stab^{tilt}(\mathbb{P}^{3})$ defined
by $(\alpha, \beta) \rightarrow (\mathrm{Coh}^{\beta}(X), Z_{\alpha,\beta})$ is continuous. Moreover, walls with respect to a class $v$ in the image of this map are locally finite. In addition, the walls in the tilt-stability space are either nested semicircles or vertical lines.
\end{thm}

\begin{rmk} \label{remJH}
Note that the Jordan-H\"older factors of an object on a wall are stable along the entire wall.
\end{rmk}

For more details on Bridgeland stability conditions on $\mathbb{P}^{3}$ we refer to \cite{Sch}.

\subsection{Bayer-Macr\`{i} linearization map} \label{BMsection} Let $S$ be a K3 surface, and $v$ a primitive algebraic class in the Mukai
lattice with self-intersection with respect to the Mukai pairing. Let $\Stab(S)$ be the space of stability conditions on $S$. In \cite{Bri8}, Bridgeland described a connected component $\Stab^{\dag}(S)$ of $\Stab(S)$ which admits a chamber decomposition.

In \cite{BM}, Bayer and  Macr\`{i} proved the following Theorem:

\begin{thm} [{\cite[Theorem 1.1]{BM}}] \label{BM}
 Let $\sigma, \eta $ be generic stability conditions with respect to $v$. Then the two moduli spaces $\mathcal{M}_{\sigma}(v)$ and $\mathcal{M}_{\eta}(v)$ of Bridgeland-stable objects are birational to
each other.

\end{thm}

As a consequence, we can canonically identify the Néron-Severi groups of  $\mathcal{M}_{\sigma}(v)$ and $\mathcal{M}_{\eta}(v)$.
Now consider the chamber decomposition of $\Stab^{\dag}(S)$ with respect to $v$ as above, and let $C$ be a
chamber. The main result of \cite{bm14} gives a natural map

$$ l_{C}\colon C \rightarrow \NS (\mathcal{M}_{C}(v)),$$
to the Néron-Severi group of the moduli space, whose image is contained in the ample cone of
$\mathcal{M}_{C}(v)$. Their main result describing the global behavior of this map is the following:

\begin{thm}[{\cite[Theorem 1.2]{BM}}] \label{BM2}
 Fix a base point $\sigma \in \Stab^{\dag}(S)$.

(a) Under the identification of the Néron-Severi groups, the maps $l_{C}$ glue to a piece-wise analytic continuous map
$$ l\colon \Stab^{\dag}(S) \rightarrow \NS (\mathcal{M}_{\sigma}(v)).$$

(b) The map $l$ is compatible, in the sense that for any generic $\sigma' \in \Stab(S)^{\dag}$, the moduli
space $\mathcal{M}_{\sigma'}(v)$ is the birational model corresponding to $l(\sigma')$. In particular, every smooth
K-trivial birational model of  $\mathcal{M}_{\sigma}(v)$ appears as a moduli space  $\mathcal{M}_{C}(v)$ of Bridgeland
stable objects for some chamber $C \subset \Stab^{\dag}(S)$.

\end{thm}

Claim (b) is the precise version of their claim that MMP can be run via wall-crossing: any minimal model can be reached after wall-crossing as a moduli space of stable objects. Extremal contractions arising as canonical models are given as coarse moduli spaces for
stability conditions on a wall.

\section{Walls and chambers} \label{Walls and chambers}

According to Theorem \ref{WC} there is  a wall-chamber structure in the stability manifold. In this section, we numerically describe the walls in $\Stab^{tilt}(\mathbb{P}^{3})$ and $\Stab(\mathbb{P}^{3})$  with respect to $\mathrm{ch}(\mathcal{I}_C)$,where $C$ is a canonical genus four curve, and give a geometric description of the first three walls. Furthermore, we describe the moduli spaces associated to the adjacent chambers.

First of all, we compute the Chern character of $\mathcal{I}_{C}$:

\begin{prop}\label{prop:prop_1}For a canonical genus 4 curve $C$ in $\mathbb{P}^{3}$, we have $\mathrm{ch}(\mathcal{I}_{C})=(1,0,-6,15)$.

\end{prop}
\begin{proof}

Since $C$ is a (2,3)-complete intersection, we have the short exact sequence $\mathcal{O}(-5) \hookrightarrow \mathcal{O}(-2) \oplus \mathcal{O}(-3) \twoheadrightarrow \mathcal{I}_{C}$, from which the claim follows.
\end{proof}

Consider the hyperbola $\H$ in the $(\alpha, \beta)$-plane defined by $\mathrm{Im}(Z_{\alpha, \beta, s}(v))=0$. For such $\alpha, \beta$, and $s$, semistable objects of Chern character $v$ have phase $0$. Moreover, these semistable objects have positive and negative phases with respect to the stability conditions on the left and right side of the hyperbola, respectively. Thus on the left side of the hyperbola,  we work with  $\mathrm{Coh}^{\beta}(\mathbb{P}^3)$ and $\mathrm{Coh}^{\alpha,\beta}(\mathbb{P}^3)$, for tilt and Bridgeland stability, respectively, and  on the right side, we work in $\mathrm{Coh}^{\beta}(\mathbb{P}^3)[-1]$ and $\mathrm{Coh}^{\alpha,\beta}(\mathbb{P}^3)[-1]$. Theorem \ref{WC} gives an order for the walls or semicircles in $\Stab^{tilt} (\mathbb{P}^{3})$. We refer to the semicircle with the smallest radius as the first wall, and so on. First, we have a Lemma:

\begin{lem} \label{lem:lem_7}   

Let $\beta$ be an integer, and $E$ a tilt semistable object in $\mathrm{Coh}^{\beta}(\mathbb{P}^{3})$.

1. If $\mathrm{ch}^{\beta}(E)=(1,1,d,e)$, then $d-1/2 \in \mathbb{Z}_{\leq 0}$. If $d=1/2$, then $E \cong \mathcal{I}_{Z}(\beta +1)$ for a zero dimensional subscheme $Z$ in $\mathbb{P}^{3}$ of length $1/6-e$. If $d=1/2-D$ where $D=1,2$, then we have $E \cong \mathcal{I}_{C_D}(\beta +1)$ where $C_D$ is a rational degree $D$ curve, plus $D-e-5/6$ (floating/embedded) points in $\mathbb{P}^{3}$.

2. If $\mathrm{ch}^{\beta}(E)=(0,1,d,e)$, then $d+1/2 \in \mathbb{Z}$ and $E \cong \mathcal{I}_{Z/P}(\beta +d +1/2)$ in which $Z$ is a zero dimensional subscheme supported in a plane in $\mathbb{P}^{3}$ and of length $1/24+d^{2}/2-e$.

\end{lem}

\begin{proof}
The first two cases of $1$ and $2$ were proven in \cite[Lemma 5.4]{Sch}. For $C_{D}$, we notice that the $\mathrm{ch}_{2}$ of an ideal sheaf of a curve is equal to  $-\mathrm{deg}$ of the curve. Also,  $\mathrm{ch}_{3}$ can be computed using Riemann-Roch. 

\end{proof}

\begin{prop}
 
 \label{Lemma 9}

Fix the class $v=(1,0,-6,15)$. The walls in $\Stab^{tilt}(1,0,-6,15)$ with respect to $v$ and for $\beta < 0$ are given by the following equations of semicircles in the $(\beta, \alpha)$ plane, with $\mathrm{ch}^{-4}_{\leq 2}$ of either the sub-object or the quotient, $F$, given as follows:
\\
1)  $(\beta +4)^{2}+\alpha^{2}=4$, $\mathrm{ch}^{-4}_{\leq 2}(F)= (1,2,2)$,
\\
2)  $(\beta + 4.5)^{2}+\alpha^{2}=8.25$, $\mathrm{ch}^{-4}_{\leq 2}(F)= (1,3,5/2)$,
\\
3) $(\beta + 5.5)^{2}+\alpha^{2}=18.25$,  $\mathrm{ch}^{-4}_{\leq 2}(F)= (1,3,7/2)$,
\\
4)  $(\beta + 6.5)^{2}+\alpha^{2}=30.25$,  $\mathrm{ch}^{-4}_{\leq 2}(F)= (1,3,9/2)$.

Furthermore, the hyperbola which is defined by $\Real(Z_{\alpha,\beta}(\nu)=0)$ where $v=\mathrm{ch}(\mathcal{I}_{C})$ intersects all these semicircles at their top.
\end{prop}

\begin{proof}

Given a short exact sequence in  $\mathrm{Coh}^{-4}(\mathbb{P}^{3})$ that defines a wall for $M^{tilt}_{\alpha, \beta}(v)$, either the subobject or the quotient of $E \in M^{tilt}_{\alpha, \beta}(v)$ will have positive rank. Let $F$ be this object, $G$ the other one, and write $\mathrm{ch}^{-4}_{\leq 2}(F)= (r,c,d)$ with $r \geq 1$. As $\mathrm{ch}^{-4}_{\leq 2}(E)= (1,4,2)$ and $F, G \in \mathrm{Coh}^{-4}(\mathbb{P}^{3})$, we have $c\geq 0$ and $4-c \geq 0$. Now, if either $c=0$ or $c=4$, then either $F$ or $G$ would have slope $+\infty$, a contradiction; Therefore $1 \leq c \leq 3$. We want to find all the possibilities for $\mathrm{ch}^{-4}_{\leq 2}(F)$ and $\mathrm{ch}^{-4}_{\leq 2}(G)$. The equation $\nu_{\alpha, -4}(E)=\nu_{\alpha, -4}(F)$ implies  $\alpha^{2}=(8d-4c)/(4r-c)$ which has to be positive. As by our assumption $r \geq 1$, and also $c \leq 3$, this implies  $d>c/2$, i.e., $d \geq c/2+1$. Combined with the Bogomolov-Gieseker inequality (Theorem \ref{BG}), we get

  \begin{equation*}
    \begin{cases}
     c^{2}\geq 2rd \geq r(c+2), \\
        r \geq 1.

    \end{cases}
  \end{equation*}
This has no solution for $c=1$. For $c=2$, the only solution is $d=2$, $r=1$. For $c=3$, we have $r=1$ and $d \in \{5/2, 7/2, 9/2\}$. Plugging these into the equation $\nu_{\alpha, -4}(E)=\nu_{\alpha, -4}(F)$ gives the corresponding semicircles. The last part comes from Bertram’s Nested Wall Theorem (Theorem \ref{WC}) which is restated in {\cite[Theorem 3.3]{Sch}} as well.
\end{proof}

At this point, we need the following Lemma:

\begin{lem} \label{(0, >)}
  Let $E \in Coh^{\beta}(\mathbb{P}^{3})$ be a $\nu_{\alpha, \beta}$-semistable object with $\mathrm{ch}(E)=(0,2,-8,e)$. Then $e \leq 49/3$. Moreover, if the equality holds, then $E=  \mathcal{O}_{ Q}(-3)$ for a (possibly singular) quadric surface $Q$ in $\mathbb{P}^{3}$.
\end{lem}
\begin{proof}
The first part is a special case of {\cite[Theorem 2.20]{Sch18}}. For the second part, the proof of {\cite[Theorem 2.20]{Sch18}} shows that in the case of equality, E becomes unstable at the wall with radius one with same center as the first wall, i.e., the wall $(\beta +4)^{2}+\alpha^{2}=1$, and also the destabilizing subobject of $E$ must have rank one. Therefore, the destabilizing short exact sequence is of the form $\mathcal{O}(-3) \hookrightarrow E \twoheadrightarrow \mathcal{O}(-5)[1]$, and so we will have $E=  \mathcal{O}_{ Q}(-3)$ for a quadric surface $Q$ in $\mathbb{P}^{3}$.
\end{proof} 

Before describing the walls in the space of Bridgeland stability conditions, we need the following result and Remarks (where $H_{\beta}^{i}(E)$'s are the cohomology objects in $ \mathrm{Coh}^{  \beta}(\mathbb{P}^{3})$):

\begin{lem}[{\cite[Lemma 8.9]{BMS}}] \label{BMS8.9}
 Let $E \in \mathrm{Coh}^{ \alpha, \beta}(\mathbb{P}^{3})$ be a $\lambda_{\alpha, \beta, s}$-semistable object, for all $s\gg1$ sufficiently big. Then it satisfies one of the following conditions:
 \\(a) $H_{\beta}^{-1}(E)=0$ and  $H_{\beta}^{0}(E)$ is $\nu_{\alpha, \beta}$-semistable.
 \\(b)  $H_{\beta}^{-1}(E)$ is $\nu_{\alpha, \beta}$-semistable and  $H_{\beta}^{0}(E)$ is either $0$ or supported in dimension $0$. 
 Moreover, if $H_{\beta}^{-1}(E)$ is $\nu_{\alpha, \beta}$-stable, $H_{\beta}^{0}(E)$ is either $0$ or zero dimensional torsion sheaf, and $\mathrm{\mathrm{Hom}}(\mathcal{O}_{p}, E)=0$ for all points $p \in \mathbb{P}^{3}$, then $E $ is  $\lambda_{\alpha, \beta, s}$-stable, for all $s\gg1$ sufficiently big. 
\end{lem}

\begin{rmk} \label{explanation}
Note that (a) and (b)  correspond to  semistable objects with respect to the stability conditions on the left and right of the hyperbola $\H$, respectively. 
\end{rmk}

Lemma \ref{BMS8.9} and Remark \ref{explanation} allow us to transfer everything from $\Stab^{tilt}(\mathbb{P}^{3})$ to $\Stab(\mathbb{P}^{3})$.

\begin{thm} \label{thm3.10}
For the  walls    $(\beta +4)^{2}+\alpha^{2}=4$ and $(\beta + 4.5)^{2}+\alpha^{2}=8.25$  in $\Stab^{tilt}$ with respect to $v$, all (tilt and Bridgeland) strictly semistable objects of Chern character $v$ can be obtained from the extensions of the  pairs of tilt (and Bridgeland) semistable objects   $\langle \mathcal{O}(-2), \mathcal{O}_{ Q}(-3)\rangle$ and $\langle \mathcal{I}_{C_{2}}(-1), \mathcal{O}_{P}(-4)\rangle$, respectively, where   $P$ is a plane in $\mathbb{P}^{3}$, $C_{2}$  a conic, and $Q$ a (possibly singular) quadric surface in $\mathbb{P}^{3}$.
\end{thm}
\begin{proof}
For the wall $(\beta + 4.5)^{2}+\alpha^{2}=8.25$, Proposition \ref{Lemma 9} implies $\mathrm{ch}^{-4}(F)=(1,3,5/2,e)$ and $\mathrm{ch}^{-4}(G)=(0,1,-1/2,5/3-e)$. The former implies $\mathrm{ch}^{-2}(F)=(1,1,-3/2,e-1/3)$. As the wall intersects the line $\beta=-2$, Remark \ref{remJH} and Lemma \ref{lem:lem_7} implies $F \cong \mathcal{I}_{C_{2}}(-1)$, where $C_{2}$ is a degree two curve plus $7/6 -e +1/3= 3/2 -e$ points. Lemma \ref{lem:lem_7} shows $G \cong \mathcal{I}_{Z/P}(-4+(-1/2)+1/2)=\mathcal{I}_{Z/P}(-4)$, in which $Z$ is a zero dimensional sub-scheme of length $e-3/2$, and $P$ is a plane in $\mathbb{P}^{3}$, containing $Z$. Non-negativity of lengths implies $e=3/2$, and so  $F =
\mathcal{I}_{C_{2}}(-1)$ and $G=\mathcal{O}_{P}(-4)$, in which $C_{2}$ is a connected conic Cohen-Macaulay curve. 

For the wall $(\beta +4)^{2}+\alpha^{2}=4$, we have $\mathrm{ch}^{-4}(F)=(1,2,2,e)$, and thus $\mathrm{ch}^{-3}(F)=(1,1,1/2,e-7/6)$. Using Lemma \ref{lem:lem_7}   implies $F \cong \mathcal{I}_{Z}(-2)$, in which $Z$ is a zero dimensional sub-scheme of length $1/6-e+7/6=4/3-e$. As the length is a non-negative integer, we have $e \leq 4/3$. On the other hand, $\mathrm{ch}(G)=(0,2,-8,53/3-e)$. Applying Lemma \ref{(0, >)}  implies $53/3-e \leq 49/3$ or $e \geq 4/3$. Therefore $e=4/3$ and $Z$ is a zero dimensional sub-scheme of length $0$,  i.e., $F \cong \mathcal{O}(-2)$. We have $\mathrm{ch}(G)=(0,2,-8,49/3)$, and hence Lemma \ref{(0, >)} implies the result.

\end{proof}

 \begin{lem} [{\cite[Corollary 11.4]{Huy}}]  \label{lemhuy} \label{exttrihuy}
  Let $j\colon Y\hookrightarrow X$ be a smooth hypersurface. Then for any $\mathcal{F} \in \mathrm{D}^{b}(Y)$ there exists a distinguished triangle
  $$\mathcal{F} \otimes \mathcal{O}_{Y}(-Y)[1] \rightarrow j^{*}j_{*}\mathcal{F} \rightarrow \mathcal{F}.$$
  \end{lem}

\begin{lem} [{\cite[Corollary 3.34]{Huy}}]  \label{lemhuyf!} 

Let $f \colon X \to Y$  be a morphism of smooth schemes over a field $k$. For any $\mathcal{F} \in \mathrm{D}^{b}(X)$ and  $\mathcal{G} \in \mathrm{D}^{b}(Y)$ there exists a functorial isomorphism (which is called Grothendieck-Verdier duality)

$$f_*Hom(\FF, f^! \GG)=Hom(f_* \FF, \GG).$$

  \end{lem}

\begin{cor} \label{duallem}
Let $\iota_P \colon P \hookrightarrow \P^3$  be the inclusion map. For any $\mathcal{F} \in \mathrm{D}^{b}(P)$ and  $\GG \in \mathrm{D}^{b}(\P^3)$, we have $(\iota_{P_* }\FF)^{\vee}=\iota_{P_*} \FF^{\vee} (1)[-1]$ and $(\iota^{*}_P \GG)^{\vee}= \GG^{\vee}$.
\end{cor}
\begin{proof}
Apply Lemma \ref{lemhuyf!} to the inclusion map.
\end{proof}

Before describing the other walls in the space of Bridgeland stability conditions, we have the following Lemmas for more general curves; these are needed to constrain the potential
destabilizing objects.

\begin{lem} \label{Gfactor}
 Fix a vector $w=(0,1,-i-1/2, e)$. The stable objects of class $w$ for stability conditions with $\mathrm{Im}(Z_{\alpha, \beta, s}(w))<0$ near the hyperbola  $\Imm(Z_{\alpha, \beta, s}(w))=0$ are of the form $\iota_{P_{*}}(\II^\vee_Z(- i))$, where $Z$ is a zero dimensional subscheme of length $l=1/6+(i/2)(i+1)-e$, and $P$ is a plane. 
\end{lem}

\begin{proof}
 Let $E$ be such an object of Chern character $w$, i.e., $E$ is (semi)stable near the hyperbola, which means it is still (semi)stable on the hyperbola  when we reach the hyperbola from the right or left. For objects with $\Imm Z_{\alpha, \beta, s} = 0$, semistability doesn't change as $s$ varies; in particular, we can let $s \to +\infty$ and apply Lemma \ref{BMS8.9} to $E[1]$. Therefore, $\HH^0_\beta(E)$ is $\nu_{\alpha, \beta}$-semistable, and $\HH^1_\beta(E)$ is a torsion sheaf $T$ with zero-dimensional support of length $m$. Notice that we have $\ch(\HH^0_\beta(E))=(0,1,-i-1/2,e+m)$. Thus, from Lemma \ref{lem:lem_7}, and varying  $\alpha, \beta$  along the wall we have $\HH^0_\beta(E)=\mathcal{I}_{Z/P}(-i)$, where $Z$ is zero dimensional subscheme of $P$. Notice that $E$ is semistable for $\Imm Z_{\alpha, \beta, s}(w) < 0$; this implies  $\Hom(\OO_Z[-1], E) = 0$, as $\OO_Z[-1]$ is semistable of phase $0$. Therefore
$\Hom(\OO_Z[-1], \II_{Z/P}(-i)) = 0$, which is possible only if $Z = \emptyset$, and thus $\HH^0_\beta(E) = \OO_P(-i)$. Having $\HH^0_\beta$ and $\HH^1_\beta$, E fits into a short exact sequence $\OO_P(-i) \to E \to T[-1]$. Dualizing this sequence, noting that $(\OO_P(-i))^\vee = \OO_P(i+1)[-1]$ (by Corollary \ref{duallem}), and letting $T':=T^\vee[3]$, we get an exact triangle
$T'[-2] \to E^\vee \to \OO_P(i+1)[-1]$. Hence we have $E^\vee[1] = (\OO_P(i+1) \to T')$. Now, $\Hom(\OO_p[-1], E) = 0$, for all $p \in P$ implies $\Hom(E^\vee[1], \OO_p[-1]) = 0$, which is equivalent to the map $\OO_P(i+1) \rightarrow T'$ being surjective. This implies $T' = \OO_Z$ for a 0-dimensional subscheme $Z$ of  $P$, and hence $E^\vee[1] = 
\iota_{P_* }\II_{Z/P}(i+1)$, where $\iota_P\colon P \hookrightarrow \P^3$. Now, Corollary \ref{duallem} implies $E = \iota_{P_* }(\II^\vee_Z(-i))$. As for the length of $Z$, using Lemma \ref{lem:lem_7} implies $l=1/6+(i/2)(i+1)-e$.

\end{proof}

\begin{lem} \label{Ffactor}
 Fix a vector $w=(1,-1,-D+1/2, e)$ for $D=1,2$. The stable objects of Chern character $w$ for stability conditions with $\mathrm{Im}(Z_{\alpha, \beta, s}(w)) <0$ near the hyperbola  $\Imm(Z_{\alpha, \beta, s}(w))=0$ are complexes $\OO_{\P^3} \to \FF$ given by the section of pure one-dimensional sheaf $\FF$ supported on a curve of degree $D$ with cokernel of length $l=3D-e-7/6 \in \mathbb{Z}_{\geq 0}$. 
\end{lem}

\begin{proof}
Let $E$ be such an object of Chern character $w$. A similar argument as in the proof of Lemma \ref{Gfactor} implies $\HH^1_\beta(E)$ is a torsion sheaf $T$ with zero-dimensional support of length $l$, $\HH^0_\beta(E)$ is $\nu_{\alpha, \beta}$-semistable,  and  $\ch(\HH^0_\beta(E))=(1,-1,-D+1/2,e+l)$. Therefore, from Lemma \ref{lem:lem_7}, and varying  $\alpha, \beta$  along the wall we have $\HH^0_\beta(E)=\mathcal{I}_{C_D}(-1)$ where $C_D$ is a rational degree $D$ curve possibly with $3D-e-l-7/6$ embedded or floating points. With the same reasoning as in the proof of Lemma \ref{Gfactor}, we have  $\Hom(\OO_{p}[-1], E) = 0$, for all $p \in \P^3$, and so $\Hom(\OO_{p}[-1], \II_{C_D}(-1)) = 0$; if $3D-e-l-7/6\neq 0$, then we have $\Hom(\OO_{Z_{3D-e-l-7/6}}[-1], \II_{C_D}(-1)) \neq 0$, which is a contradiction, and hence we must have $l=3D-e-7/6$. Therefore, $E$ fits into 
$\mathcal{I}_{C_D}(-1) \rightarrow E \rightarrow T[-1]$, with $\dim(T)=l$,
and hence by definition of stable pairs,  we  get the result.

\end{proof}

\begin{prop}
 For the semicircle  $(\beta + 5.5)^{2}+\alpha^{2}=18.25$, the corresponding walls in the Bridgeland stability space on the right of the left branch of the hyperbola $\H$ are given by three sets of objects   $\langle \mathcal{I}_{L}(-1),\iota_{P_{*}}(\mathcal{I}_{Z_{2}})^{\vee}(-5)\rangle$, $\langle (\mathcal{O}(-1) \rightarrow \mathcal{O}_{L}),\iota_{P_{*}}(\mathcal{I^{\vee}}_{Z_{1}})(-5)\rangle$, and $\langle (\mathcal{O}(-1) \rightarrow \mathcal{O}_{L}(1)), \mathcal{O}_{P}(-5)\rangle$, where $L$ is a line, and $Z_{i}$ is length $i$ zero-dimensional subscheme contained in $P$.
\end{prop}

\begin{proof}
From Proposition \ref{Lemma 9}, we have $\mathrm{ch}^{-4}_{\leq 2}(F)= (1,3,7/2)$ for a subobject or quotient of semistable objects with respect  stability conditions on the semicircle. Therefore we have  $\mathrm{ch}(F)= (1,-1,-1/2,e)$. Lemma \ref{Ffactor} implies  $F \cong (\mathcal{O}(-1) \rightarrow \mathcal{O}_{L}(l-1))$, where $L$ line, with $l=11/6-e \in \mathbb{Z}^{\geq 0}$  points on it. On the other hand, we can see that the other corresponding JH factor is given by $\mathrm{ch}(G)= (0,1,-11/2,15-e)$. Now using Lemma \ref{Gfactor}, we have $G \cong \iota_{P_{*}}(\mathcal{I^{\vee}}_{Z_{l'}})(-5)$, where $Z_{l'}$ is a zero dimensional subscheme of length $l'=1/6+e \in \mathbb{Z}^{\geq 0}$. The conditions on the lengths $l, l'$ gives three possibilities for $e$, i.e., $e \in \{-1/6,5/6,11/6\}$. This gives the three pairs stated above.

\end{proof}

Now, let us describe the chambers close to the hyperbola.

\begin{figure}[h]
 \centering
    \includegraphics[width=7.2cm]{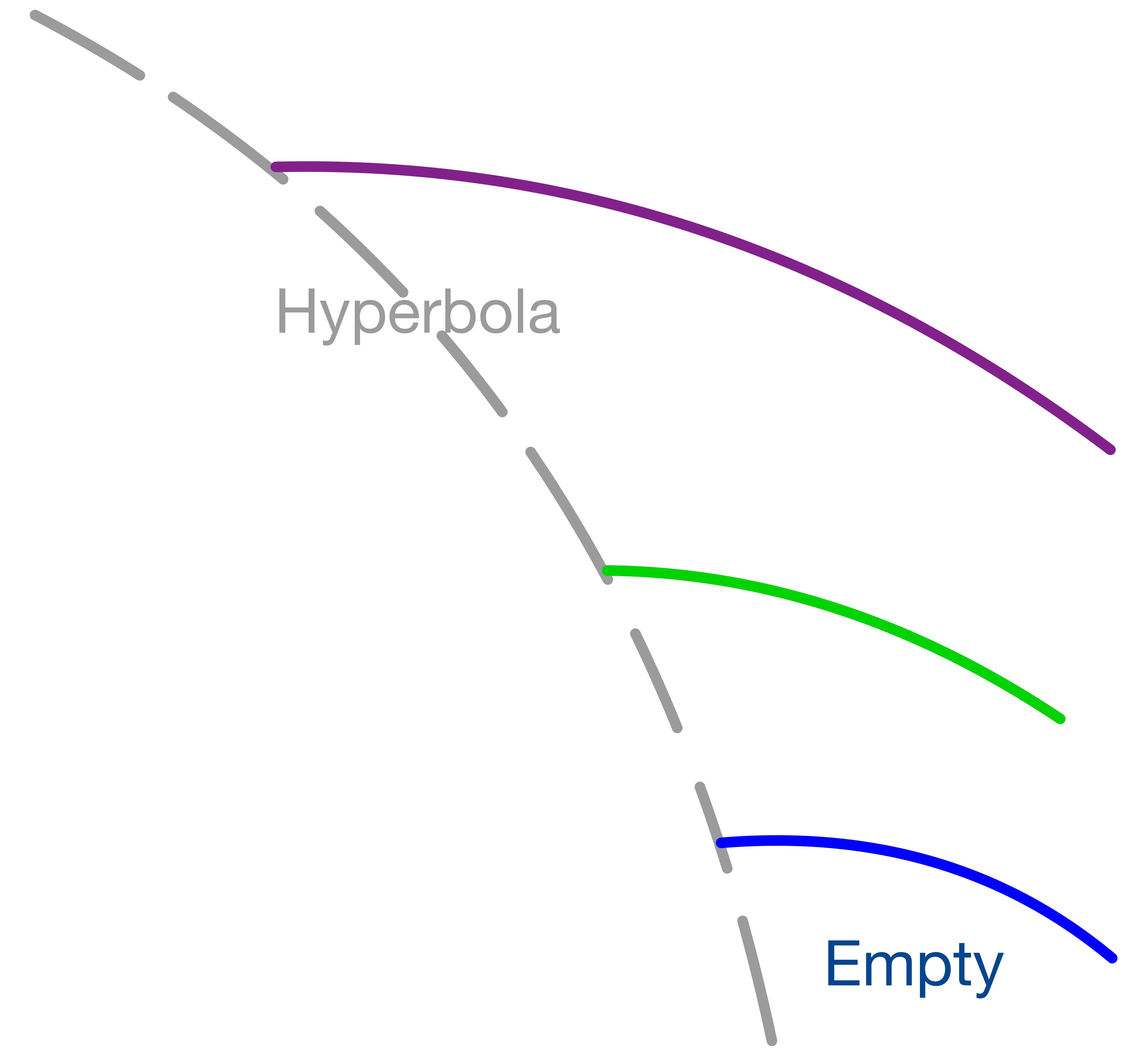}
    \caption{The first three walls on the right of the left branch of the hyperbola}

\end{figure}

\begin{prop}
The moduli space for the first chamber below the smallest wall, which is given by all the extensions of two objects $\mathcal{O}(-2), \mathcal{O}_{ Q}(-3)$, is empty.
\end{prop}

\begin{proof}

These two objects define the semicircle  $(\beta +4)^{2}+\alpha^{2}=4$ (Proposition \ref{Lemma 9})  which is above the semicircle $(\beta +3.75)^{2}+\alpha^{2}=2.06$ given by the BMT inequality (Theorem \ref{BMT14:stability_threefolds}). We know that the moduli space is empty below the BMT semicircle. On the other hand, the moduli space can become non-empty only as we cross a wall as below the BMT semicircle there is no stable object. Combining these two, we conclude the claim.
\end{proof}

Let $\mathcal{M}_{1}$ be the first non-empty moduli space for the next chamber, which appears after crossing the smallest wall $\langle \mathcal{O}(-2), \mathcal{O}_{ Q}(-3)\rangle$. In the following Proposition we will see that this moduli space gives a compactification of objects corresponding to $(2,3)$-complete intersection curves in $\mathbb{P}^{3}$:

\begin{prop} \label{N1}
The moduli space $\mathcal{M}_{1}$ is a $\mathbb{P}^{15}$-bundle over  $\mathbb{P}^{9}$. More precisely, the complement of (2,3)-complete intersections in $\mathcal{M}_{1}$ are parametrized by the pairs $(Q, C)$ where $Q=P \cup P'$ is a union of two planes and $C$ is a conic in one of the two planes. The associated objects are non-torsion free sheaves $E$  given by 
$\mathcal{O}_{P}(-4) \hookrightarrow E \twoheadrightarrow \mathcal{I}_{C_{2}}(-1),$ for a conic $C_{2}$.
\end{prop}

\begin{proof}
Any object $E$ in $\mathcal{M}_{1}$ is generated by a short exact sequence $\mathcal{O}(-2) \hookrightarrow E \twoheadrightarrow \mathcal{O}_{Q}(-3)$. The parameter space of $\mathcal{O}(-2)$ is just given by a point, whereas objects of the form $\mathcal{O}_{Q}(-3)$ are parametrized by  $\mathbb{P}(\mathrm{H}^{0}(\mathbb{P}^{3}, \mathcal{O}(2)))=\mathbb{P}(\mathbb{C}^{10})=\mathbb{P}^{9}$. Given $Q$, the extensions parametrized by $\mathrm{H}^{0}(Q,\mathcal{O}(3))=\mathbb{C}^{16}$ up to rescaling, i.e., $\mathbb{P}^{15}$. Therefore $\mathcal{M}_{1}$ will be a $\mathbb{P}^{15}$-bundle over $ {\{pt}\}  \times \mathbb{P}^{9} \cong \mathbb{P}^{9}$.
Notice that a cubic surface, $S$, and a quadric, $Q$, define a complete intersection curve if and only if $S$ does not contain any irreducible component of $Q$. Thus being (2,3)-complete intersection is an open condition. Let us denote by CI the set of (2,3)-complete intersections in $\mathcal{M}_{1}$. We notice that $\mathcal{M}_{1}$ compactifies CI with the pairs of a quadric $Q$, plus a cubic equation that does not vanish entirely on $Q$; i.e., it can be zero on one of the two components of $Q$, when $Q$ is reducible. Therefore in the complement $\mathcal{M}_{1}\setminus CI$ we have $(S,Q)$ such that $Q=P \cup P'$ is a union of two planes. This induces the sequence $\mathcal{O}_{P}(-1) \hookrightarrow \mathcal{O}_{Q} \twoheadrightarrow \mathcal{O}_{P'}$. Now if $C_2$ is a conic in $P'$, we have $\OO(-1)=\II_{P'} \hookrightarrow \II_{C_2} \twoheadrightarrow \OO_{P'}(-2)$. Combining these two sequences with the defining sequence of $E$, we get the lift of $\mathcal{O}_{P}(-4) \hookrightarrow \mathcal{O}_{Q}(-3)$ to the subobject of the claimed extension as follows:

   \begin{center}
    
\begin{tikzcd}[column sep=
scriptsize
]
 0 \arrow[r, hook, ""]  \arrow[d, hook ] & \OO_P(-4) \arrow[r, dash, "\cong"] \arrow[d, hook, dashed ]&  \mathcal{O}_{P}(-4)    \arrow[d,  hook, "" ]\\
\mathcal{O}(-2)   \arrow[r,  ""]  \arrow[d, dash, "\cong" ]  &E\arrow[r, ""]  \arrow[d, two heads, dashed ] 
&  \mathcal{O}_{Q}(-3)   \arrow[d, two heads ] \\
\mathcal{O}(-2) \arrow[r, hook, ""]  & \II_{C_2}(-1) \arrow[r, two heads, ""] &
 \mathcal{O}_{P'}(-3).
\end{tikzcd}
\end{center}
 As $\mathrm{Ext}^{1}(\mathcal{I}_{C_{2}}(-1), \mathcal{O}_{P}(-4))= \mathbb{C}$, we will have the result.

\end{proof}

To show that $\mathcal{M}_{1}$ is a blow-up
in Proposition \ref{N2}, we need the following lemmas:

\begin{lem}[{\cite[Lemma 4.4]{GHS}}] \label{3.13}
Let $F \hookrightarrow E \twoheadrightarrow G$ be an exact sequence at a wall in Bridgeland stability with $E$ semistable to one side of the wall and $F, G$ distinct stable objects of the same (Bridgeland) slope. Then we have:
$$\mathrm{ext}^{1}(E,E) \leq \mathrm{ext}^{1}(F,F)+\mathrm{ext}^{1}(G,G)+ \mathrm{ext}^{1}(F,G)+ \mathrm{ext}^{1}(G,F)-1.$$
\end{lem}

\begin{lem}[{\cite{Moi}, ~\cite[Theorem 4.6]{GHS}}] \label{3.15}
Any birational morphism $f\colon X \rightarrow Y$ between smooth proper algebraic spaces of finite type over complex numbers such that the contracted locus $E$ is irreducible, and $f(E)$ is smooth, is the blow up of $Y$ in $f(E)$. 
\end{lem}

Now, we want to describe the moduli space $\mathcal{M}_{\sigma_{-}}(v)$ for the next chamber,  which comes after the wall $\langle\mathcal{I}_{C_{2}}(-1), \mathcal{O}_{P}(-4)\rangle$. Since $\mathrm{Ext}^{1}(\mathcal{I}_{C_{2}}(-1),\mathcal{O}_{P}(-4))=\mathbb{C}$ for all $\mathcal{I}_{C_{2}} \in {\mathcal{H}ilb^{2t+1}(\mathbb{P}^{3})}$ and all hyperplanes $P \in (\mathbb{P}^{3})^{\vee}$, there is a unique object in $\mathcal{M}_{1}$ destabilized at the wall $\langle\mathcal{I}_{C_{2}}(-1), \mathcal{O}_{P}(-4)\rangle$, identifying the destabilized locus with ${\mathcal{H}ilb^{2t+1}(\mathbb{P}^{3})} \times (\mathbb{P}^{3})^{\vee}$.

\begin{prop} \label{N2}
The moduli space $\mathcal{M}_{\sigma_{-}}(v)$ for the next chamber is a blow up of $\mathcal{M}_{1}$ in the locus ${\mathcal{H}ilb^{2t+1}(\mathbb{P}^{3})} \times (\mathbb{P}^{3})^{\vee}$.
\end{prop}
\begin{proof}

We assume that $P\neq P'$ (otherwise, the relevant Ext-groups vanish). Then, a simple computation shows  $\mathrm{Ext}^{1}(\mathcal{I}_{C_{2}}(-1),\mathcal{O}_{P}(-4))=\mathbb{C}$, and $\mathrm{Ext}^{1}(\mathcal{O}_{P}(-4),\mathcal{I}_{C_{2}}(-1))\\=\mathbb{C}^{13}$. We know that the parameter space for $\mathcal{I}_{C_{2}}(-1)$ is ${\mathcal{H}ilb^{2t+1}(\mathbb{P}^{3})}$, and the parameter space for $\mathcal{O}_{P}(-4)$ is $(\mathbb{P}^{3})^{\vee}$, and we have $dim (\mathcal{H}ilb^{2t+1}(\mathbb{P}^{3}))=8$, and $dim ((\mathbb{P}^{3})^{\vee})=3$. Therefore the locus of extensions  $\mathcal{O}_{P}(-4)\to E \to \mathcal{I}_{C_{2}}(-1)$ is isomorphic to ${\mathcal{H}ilb^{2t+1}(\mathbb{P}^{3})} \times (\mathbb{P}^{3})^{\vee}$. Lemma \ref{3.13} implies $\mathrm{ext}^{1}(E,E) \leq 8+3+1+13-1=24$ for each $E$ given by a class in $\mathrm{Ext}^{1}(\mathcal{O}_{P}(-4),\mathcal{I}_{C_{2}}(-1))$. This implies that $\mathcal{M}_{\sigma_{-}}(v)$ is smooth. The dimension of the locus of extensions in $\mathrm{Ext}^{1}(\mathcal{O}_{P}(-4),\mathcal{I}_{C_{2}}(-1))$ (i.e., the exceptional locus)  is $12+(8+3)=23$, so it is a divisor in $\mathcal{M}_{\sigma_{-}}(v)$. We also note that $\mathcal{M}_{\sigma_{-}}(v)$ is connected. Now, using Lemma  \ref{3.15} induces the result.

\end{proof}

\section{Exceptional locus} \label{Exceptional locus}

 In this section, we  describe the moduli space $\mathcal{M}_{\sigma_{+}}(v)$ for the next chamber as $\mathcal{M}_{\sigma_{+}}(v)=\widetilde{\mathcal{M}_{\sigma_{-}}(v)} \cup \mathcal{M}^{'}$, where $\widetilde{\mathcal{M}_{\sigma_{-}}(v)} $  is birational to $\mathcal{M}_{\sigma_{-}}(v)$, and $\mathcal{M}^{'}$ is a new irreducible component.
 
 Let $A=\mathcal{I}_{L}(-1)$ and $B=\iota_{P_{*}}(\mathcal{I}_{Z_{2}})^{\vee}(-5)$. Let  $\sigma_{0}$ be a stability condition on the wall $\WW$. As the morphism  $\mathcal{M}_{\sigma_{+}}(v) \dashrightarrow \mathcal{M}_{\sigma_{0}}(v)$ contracts $ \mathcal{M}^{'}$ (projecting the $\mathbb{P}^{17}$-bundle to its base), to understand its restriction $\widetilde{\mathcal{M}_{\sigma_{-}}(v)}   \rightarrow \mathcal{M}_{\sigma_{-}}(v) $,  we first need to understand the intersection $\widetilde{\mathcal{M}_{\sigma_{-}}(v)} \cap \mathcal{M}^{'}$; this is the main goal of this section. We can stratify the base of $\mathcal{M}'$ via
$\mathrm{Ext}^1(A, B)$ (as in Lemma \ref{ext1purple1}). We know that the intersection lies over the strata where 
$\mathrm{ext}^1(A,B) \ge 1$.

To  precisely describe the singularity locus  of $\mathcal{M}_{\sigma_{+}}(v)$ as $\ker(\delta)$, we take the following route:  we first compute necessary Ext-groups.
Then for the precise description of the singularity locus  of $\mathcal{M}_{\sigma_{+}}(v)$, we need to prove the surjectetivity of the map $$\delta\colon \mathrm{Ext}^{1}(B,A) \rightarrow \mathrm{Hom}(\mathrm{Ext}^{1}(A,B),  \mathrm{Ext}^{2}(B,B))$$
(Lemmas \ref{surj1} and \ref{surj2}).
To recover the 14-dimensional cone, we construct enough degenerations of objects in the intersection to objects in $\mathcal{M}'$. More precisely, we project the canonical genus four curves to the plane quintics with two nodes. Finally, we degenerate the quintic with two nodes union a line  to a plane quartic union a thickened line meeting those nodes (Lemma \ref{closure}). This implies that the singularity locus is the same as the intersection of the two components (which is the same as the exceptional locus of $\psi$). Having all these, Theorem  \ref{intN3} will be proved.

\subsection{Ext Groups}

To study this wall-crossing, we first compute all  necessary $\mathrm{Ext}$-groups associated to the wall  $\mathcal{W}=\langle \mathcal{I}_{L}(-1), \iota_{P_{*}}(\mathcal{I}_{Z_{2}})^{\vee}(-5)\rangle$.

First, we have the following Lemmas:

\begin{lem} \label{pullback}

 Let $P \subset \mathbb{P}^{3}$ be a plane,  $L \subset \mathbb{P}^{3}$ be an arbitrary line,  $L' \subset P$ be a line intersecting (but not identical to) $L$, and $p=P \cap L$. Then

\begin{equation*}
   \iota_{P}^{*}\mathcal{I}_{L} =
    \begin{cases}
     \mathcal{I}_{p/P}, &L \not\subset P\\
      \mathcal{O}_{P}(-1) \oplus \mathcal{O}_{L}(-1), & L \subset P.
      
      \\
  
    \end{cases}
  \end{equation*}

\end{lem}

\begin{proof}
First assume $L \not \subset P$; the exact sequence $\mathcal{I}_{L} \rightarrow \mathcal{O} \rightarrow \mathcal{O}_{L}$ implies
$\iota_{P}^{*}\mathcal{I}_{L} =\mathcal{I}_{p/P}$. Now assume $L \subset P$. Assume that $P \cap P'=L$ where $P'$ is a plane containing $L$. From the resolution $\mathcal{O}(-2) \rightarrow \mathcal{O}(-1) \oplus \mathcal{O}(-1) \rightarrow \mathcal{I}_{L}$ induced by $P$ and $P'$, we have $\iota_{P}^{*}\mathcal{I}_{L}= \iota_{P}^{*}(\mathcal{O}(-2) \rightarrow \mathcal{O}(-1)) \oplus \iota_{P}^{*}(\mathcal{O}(-1)) =(\mathcal{O}_{P}(-2) \rightarrow \mathcal{O}_{P}(-1) )\oplus \mathcal{O}_{P}(-1)= \mathcal{O}_{L}(-1) \oplus \mathcal{O}_{P}(-1) $.

\end{proof}

In the statement of the following Lemma, all tensor products are in  $\mathrm{D}^{b}(P)$:
\begin{lem} \label{tensor}
Let $P$ be a plane, $p \subset P$ a single point,   $Z_2=q \cup q' \subset P$ a point of length 2, and $l \subset P $ a line. Then $\II_{Z_2} \otimes \II_p$ fits into the short exact sequence  $\OO_p \hookrightarrow \II_{Z_2} \otimes \II_p \twoheadrightarrow{}  \II_{2p \cup q'}$, for $p=q \subset Z_2$, and 

$$\II_{Z_2} \otimes \II_p=\II_{p \cup Z_2},$$ for $ p \not \subset Z_2$. Furthermore, we have

\begin{equation*}
  \II_{Z_2} \otimes \OO_p =
    \begin{cases}
     \OO_p[1]\oplus \OO_p^{\oplus 2}, &p \subset Z_2\\
      \mathcal{O}_{p}, & p \not \subset Z_2
  
    \end{cases}
  \end{equation*}
  
  \begin{equation*}
  \II_{Z_2} \otimes \OO_l =
     \begin{cases}
      \OO_{Z_2} \oplus \OO_l(-2),&   Z \subset l\\
     \OO_{q} \oplus \OO_l(-1),& q \subset l,   \text{and $ q' \not\subset l$}\\
    
     \OO_l, &  q, q' \not \subset l

    \end{cases}
  \end{equation*}

\end{lem}

\begin{proof}
 Tensoring the short exact sequence $\mathcal{I}_{p} \hookrightarrow \mathcal{O}_{P} \twoheadrightarrow \mathcal{O}_{p}$ by $\mathcal{I}_{p}$, we have 
   
   $$0 \to \mathcal{T}or^{1}(\mathcal{I}_{p},\mathcal{O}_{p}) \rightarrow \mathcal{I}_{p}\otimes \mathcal{I}_{p} \rightarrow \mathcal{I}_{p} \rightarrow \mathcal{O}_{p} \otimes^{u} \mathcal{I}_{p} \to 0,$$
   where $\otimes^{u}$, is the underived tensor.
   We know that $\mathcal{I}_{p} \cong[ \mathcal{O}_{P}(-2) \hookrightarrow \mathcal{O}_{P}(-1)^{\oplus2}]$; tensoring this by $\mathcal{O}_{p}$ gives  $\mathcal{I}_{p} \otimes \mathcal{O}_{p} \cong [\mathcal{O}_{p} \xrightarrow{0} \mathcal{O}_{p}^{\oplus 2}]$, and so we have $\II_p \otimes \OO_p=\OO_p[1] \oplus \OO_p^{\oplus 2}$.
   
Let us consider $\II_p \otimes \II_p$. The above sequence will be
   
   $$\mathcal{O}_{p}  \rightarrow \mathcal{I}_{p}\otimes \mathcal{I}_{p} \xrightarrow{f} \mathcal{I}_{p} \xrightarrow{g} \mathcal{O}_{p}^{\oplus2}=\mathcal{I}_{p}/\mathcal{I}_{p}^{2} \rightarrow 0.
  $$
 But $im(f)=ker(g)=\mathcal{I}_{p}^{2} $, so we have
 
   \begin{center}

   \begin{tikzpicture}[>=angle 90]
\matrix(a)[matrix of math nodes,
row sep=3em, column sep=2.5em,
text height=1.5ex, text depth=0.25ex]
{\mathcal{O}_{p} &\mathcal{I}_{p}\otimes \mathcal{I}_{p}&\mathcal{I}_{p} &\mathcal{I}_{p}/\mathcal{I}_{p}^{2}
\\
&\mathcal{I}_{p}^{2}\\};
\path[->](a-1-1) edge (a-1-2);
\path[->>](a-1-2) edge (a-2-2);

\path[->](a-1-2) edge (a-1-3);
\path[right hook->](a-2-2) edge (a-1-3);
\path[->](a-1-3) edge (a-1-4);
\end{tikzpicture}
 \end{center}
Therefore, we get the short exact sequence $\OO_p \hookrightarrow \II_p \otimes \II_p \twoheadrightarrow{} \II_p^{2}$. 

For $ \II_{Z_2} \otimes \II_p$, when $p=q \sub Z_2$, considering the above computation, and looking at $\mathcal{I}_{Z_2}\otimes \mathcal{I}_{q} \rightarrow \mathcal{I}_{q}\otimes \mathcal{I}_{q} \rightarrow \mathcal{I}_{q'}$, imply the result. 
The result is immediate for  $p \not \sub Z_2$.

  For $ \II_{Z_2} \otimes \OO_p$, if $p \subset Z_2$, then tensoring  $\mathcal{I}_{Z_2} \cong [\mathcal{O}_{P}(-3) \hookrightarrow \mathcal{O}_{P}(-2) \oplus \OO_P(-1)]$  by $\mathcal{O}_{p}$, and the same argument as above give the result. 
  If $p \not \subset Z_2$, using the above, the result is immediate again.

Now, for $\II_{Z_2} \otimes \OO_l$, if $Z_2 \subset l$, tensoring $\mathcal{I}_{Z_2}$ as above by $\OO_l$ gives $\OO_l(-3) \to \OO_l(-1) \oplus \OO_l(-2)$, which is $\OO_{Z_2} \oplus\OO_l(-2)$. If $q \subset l$, but $q' \not \subset l$, the same argument gives the result. If $q, q' \not \subset l$, then tensoring $\II_l \into \OO_P \onto \OO_l$ by $\II_{Z_2}$, and noticing that $\II_{Z_2} \otimes \II_l=\II_{Z_2 \cup l}$ in this case (which can be similarly checked as above), imply the claim.

\end{proof}

Now we have the following Lemma on $\mathrm{Ext}$-groups:

\begin{lem} ~\label{ext1purple1}
For the wall $\langle \mathcal{I}_{L}(-1),\iota_{P_{*}}(\mathcal{I}_{Z_{2}})^{\vee}(-5)\rangle$, we have:

$$
\mathrm{Ext}^{1}(\mathcal{I}_{L}(-1),\mathcal{I}_{L}(-1))= \mathbb{C}^{4},\quad
\mathrm{Ext}^{1}(\iota_{P_{*}}(\mathcal{I}_{Z_{2}})^{\vee}(-5), \iota_{P_{*}}(\mathcal{I}_{Z_{2}})^{\vee}(-5))= \mathbb{C}^{7},$$$$
\mathrm{Ext}^{1}(\iota_{P_{*}}(\mathcal{I}_{Z_{2}})^{\vee}(-5),\mathcal{I}_{L}(-1))= \mathbb{C}^{18},
$$
\begin{equation*}
    \mathrm{Ext}^{1}(\mathcal{I}_{L}(-1),\iota_{P_{*}}(\mathcal{I}_{Z_{2}})^{\vee}(-5))=
    \begin{cases}
    0, & \langle Z_2 \rangle \cap L = \emptyset\\
       \mathbb{C}, & \text{ $\langle Z_2 \rangle \cap L \neq \emptyset$  but $\langle Z_2 \rangle \neq L$}\\
     
      \mathbb{C}^{2}, &  \langle Z_2 \rangle =  L

    \end{cases}
  \end{equation*}
where $\langle Z_2 \rangle$ is the line spanned by $Z_{2}$.
Furthermore, we have 
$$
\mathrm{Ext}^{2}(\mathcal{I}_{L}(-1),\mathcal{I}_{L}(-1))=0, \quad \mathrm{Ext}^{2}(\iota_{P_{*}}(\mathcal{I}_{Z_{2}})^{\vee}(-5),\mathcal{I}_{L}(-1))= 0,$$$$
\mathrm{Ext}^{2}(\iota_{P_{*}}(\mathcal{I}_{Z_{2}})^{\vee}(-5), \iota_{P_{*}}(\mathcal{I}_{Z_{2}})^{\vee}(-5))=\mathbb{C}^{4}.
$$

\end{lem}

\begin{proof}
Since $\mathrm{Ext}^{1}(\mathcal{I}_{L}(-1), \mathcal{I}_{L}(-1))$ is the tangent space of $G(2,4)$, we can see $\mathrm{Ext}^{1}(\mathcal{I}_{L}(-1),\\ \mathcal{I}_{L}(-1))= \mathbb{C}^{4}$. We know that $\mathrm{Ext}^{1}(\iota_{P_{*}}(\mathcal{I}_{Z_{2}})^{\vee}(-5), \iota_{P_{*}}(\mathcal{I}_{Z_{2}})^{\vee}(-5))$ is the tangent space to the space parametrising two points in a plane in $\mathbb{P}^{3}$. This is a bundle over $(\mathbb{P}^{3})^{*}$ with fibers isomorphic to $(\mathbb{P}^{2})^{[2]}$, thus it is smooth of dimension $7$, and hence $\mathrm{Ext}^{1}(\iota_{P_{*}}(\mathcal{I}_{Z_{2}})^{\vee}(-5), \iota_{P_{*}}(\mathcal{I}_{Z_{2}})^{\vee}(-5))\\=\mathbb{C}^{7}.$

For $\mathrm{Ext}^1(A,B)$, using Serre duality we have:
$$\mathrm{Ext}^{1}(\mathcal{I}_{L}(-1),\iota_{P_{*}}(\mathcal{I}_{Z_{2}})^{\vee}(-5))=\mathrm{Ext}^{1}(\iota_{P}^{*}(\mathcal{I}_{L}(-1)),(\mathcal{I}_{Z_{2}})^{\vee}(-5)) =\mathrm{H}^{1}(\iota_{P}^{*}(\mathcal{I}_{L}) \otimes\mathcal{I}_{Z_{2}}(1))^{\vee}.$$
We use Lemma \ref{pullback} for $\iota_{P}^{*}(\mathcal{I}_{L}(-1))$; so there are two cases:
\begin{description}

\item [1) $L \subset P$] Let $Z_2=q \cup q'$, for $q$ and $q'$ single points. By Lemma \ref{tensor}, we have
$$\mathrm{Ext}^{1}(A,B)= \mathrm{H}^{1}( \mathcal{I}_{Z_{2}})^{\vee} \oplus \mathrm{H}^{1}(\mathcal{O}_{L} \otimes \mathcal{I}_{Z_{2}})^{\vee}$$
\begin{equation*}
    =
    \begin{cases}
      \mathrm{H}^{1}( \mathcal{I}_{Z_{2}})^{\vee}\oplus \mathrm{H}^{1}( \mathcal{O}_{L}(-2) \oplus \OO_{Z_2})^{\vee}  =\mathbb{C}\oplus (\mathbb{C} \oplus 0)=\mathbb{C}^{2}, &   \text{$L \subset P$ and $Z_{2} \subset L$}\\
      \mathrm{H}^{1}( \mathcal{I}_{Z_{2}})^{\vee} \oplus \mathrm{H}^{1}((\OO_L(-1) \oplus \OO_{q})\otimes \II_{q'})^{\vee} & L \subset P, q \subset L, q' \not\subset L \\ \quad \quad \quad =\C \oplus \mathrm{H}^{1}(\OO_L(-1) \oplus \OO_{q})^{\vee} =\mathbb{C}\oplus (0+0), & \\
    
      \mathrm{H}^{1}( \mathcal{I}_{Z_{2}})^{\vee} \oplus \mathrm{H}^{1}(\OO_L)^{\vee} =\mathbb{C}\oplus 0, & \text{$L \subset P$, and $q, q' \not \subset L$}

    \end{cases}
  \end{equation*}

\item [2) $L$ is not contained in $P$]

Recall that $p=P \cap L$. In this case, again using Lemma  \ref{pullback} ($\iota_{P}^{*}(\mathcal{I}_{L})= \mathcal{I}_{p}$) and Serre duality we have $\mathrm{Ext}^{1}(A,B)= \mathrm{H}^{1}(\mathcal{I}_{p} \otimes \mathcal{I}_{Z_{2}}(1))^{\vee}.$
Now, we consider the exact triangle
$$\mathcal{I}_{p} \otimes \mathcal{I}_{Z_{2}}(1) \rightarrow  \mathcal{I}_{Z_{2}}(1) \rightarrow \mathcal{O}_{p} \otimes \mathcal{I}_{Z_{2}}(1).$$
\begin{itemize}
    \item If $p \not \subset Z_{2}$, then using Lemma \ref{tensor},  we have $ \mathcal{O}_{p} \otimes \mathcal{I}_{Z_{2}}(1)= \mathcal{O}_{p}$ and so taking the long exact cohomology sequence of the above triangle, and noticing that the unique global section of $\mathcal{I}_{Z_{2}}(1)$ vanishes exactly along $\langle Z_2 \rangle$, implies that the map $\mathrm{H}^{0}(\mathcal{I}_{Z_{2}}(1))=\mathbb{C} \rightarrow \mathrm{H}^{0}(\mathcal{O}_{p})=\mathbb{C} $ is non-zero if and only if $p$ is not colinear with $Z_{2}$. We also notice that $\mathrm{H}^0(\mathcal{I}_{p} \otimes \mathcal{I}_{Z_{2}}(1))=\mathrm{H}^0(\mathcal{I}_{p \cup Z_2} (1))=0$, by Lemma \ref{tensor}. Furthermore, we have $\mathrm{H}^{1}(\mathcal{I}_{Z_{2}}(1)) =0$. Therefore in this case, $\mathrm{Ext}^1(A,B)=\mathrm{H}^{1}(\mathcal{I}_{p} \otimes \mathcal{I}_{Z_{2}}(1))^{\vee} =\mathbb{C}$ if $p$ is colinear with $Z_{2}$, and $\mathrm{Ext}^1(A,B)=\mathrm{H}^{1}(\iota_{P}^{*}(\mathcal{I}_{L}) \otimes\mathcal{I}_{Z_{2}}(1))^{\vee} =0$ otherwise.

\item If $p \subset Z_{2}$,  then using Lemma \ref{tensor}, we have 
$\mathcal{H}^{0}( \mathcal{O}_{p} \otimes \mathcal{I}_{Z_{2}}(1))=\mathcal{I}_{Z_{2}}/m_{p}.\mathcal{I}_{Z_{2}} =(\mathcal{O}_{p})^{\oplus 2}$ and $\mathcal{H}^{-1}( \mathcal{O}_{p} \otimes \mathcal{I}_{Z_{2}}(1))=\mathcal{T}or^{1}( \mathcal{O}_{p} \otimes \mathcal{I}_{Z_{2}}(1))=\mathcal{O}_{p}$. The map $\mathrm{H}^{0}(\mathcal{I}_{Z_{2}}(1))=\mathbb{C} \rightarrow \mathrm{H}^{0}( \mathcal{O}_{p} \otimes \mathcal{I}_{Z_{2}}(1))=\mathrm{H}^{0}((\mathcal{O}_{p})^{\oplus 2})=\mathbb{C}^{2} $ is always non-zero, and as $\mathrm{H}^{1}(\mathcal{I}_{Z_{2}}(1)) =0$, we will have $\mathrm{Ext}^1(A,B)=\mathrm{H}^{1}(\mathcal{I}_{p} \otimes \mathcal{I}_{Z_{2}}(1))^{\vee} =\mathbb{C}$. 
\end{itemize}
\end{description}
Now we compute $\mathrm{Ext}^{1}(\iota_{P_{*}}(\mathcal{I}_{Z_{2}})^{\vee}(-5), \mathcal{I}_{L}(-1))$. In case  $L \subset P$
we use Lemma \ref{pullback} to obtain
$$
\mathrm{Ext}^{1}(\iota_{P_{*}}(\mathcal{I}_{Z_{2}})^{\vee}(-5), \mathcal{I}_{L}(-1))=\mathrm{Ext}^{1}((\mathcal{I}_{Z_{2}})^{\vee}(-5), \iota_{P}^{!}(\mathcal{I}_{L}(-1)))
=\mathrm{H}^{0}(\iota_{P}^{*}(\mathcal{I}_{L}) \otimes \mathcal{I}_{Z_{2}} \otimes \mathcal{O}(5))$$$$=\mathrm{H}^{0}\bigl((\mathcal{O}(-1) \oplus \mathcal{O}_{L}(-1)) \otimes \mathcal{I}_{Z_{2}}(5)\bigr)=\mathrm{H}^{0}(\mathcal{I}_{Z_{2}}(4)) \oplus \mathrm{H}^{0}(\mathcal{I}_{Z_{2}} \otimes \mathcal{O}_{L}(4))=\mathbb{C}^{13}\oplus \mathrm{H}^{0}(\mathcal{I}_{Z_{2}} \otimes \mathcal{O}_{L}(4)).
$$
Tensoring $\mathcal{O}_{\mathbb{P}^{2}}(-1) \rightarrow \mathcal{O}_{\mathbb{P}^{2}} \rightarrow \mathcal{O}_{L}$  by  $\mathcal{I}_{Z_{2}}(4) $ implies $ \mathrm{H}^{0}(\mathcal{I}_{Z_{2}} \otimes \mathcal{O}_{L}(4))=\mathbb{C}^{5}$, and therefore  $\mathrm{Ext}^{1}(B,A)=\mathbb{C}^{18}$. Now let $L \not \subset P$. As in this case we have  $\iota_{P}^{*}(\mathcal{I}_{L}) =\mathcal{I}_{p}$, a similar computation as before and using Lemma \ref{tensor} show that

\begin{equation*}
  \mathrm{Ext}^{1}(B,A)=\mathrm{H}^{0}(\mathcal{I}_{p} \otimes \mathcal{I}_{Z_{2}}(5)) =
    \begin{cases}
     \mathrm{H}^0(\OO_p)\oplus \mathrm{H}^{0}(\II_{2p \cup q'}(5))=\C \oplus \C^{17}, &p=q \subset Z\\
     \mathrm{H}^0(\II_{p \cup Z_2}(5))=\C^{18}, & p \not \subset Z_2
  
    \end{cases}
  \end{equation*}
  (notice that $2p \cup q$ has length 4 in the first case, and $p \cup Z_2$ has length 3 in the second case).

We can easily see that $\mathrm{Ext}^{2}(\mathcal{I}_{L}(-1), \mathcal{I}_{L}(-1))=0.$ 
As for $\mathrm{Ext}^{2}(B,A)$, using Lemma \ref{tensor}, in a similar way as for $\mathrm{Ext}^{1}(B,A)$, we can see that when $L \subset P$, we have $$
\mathrm{Ext}^{2}(B,A)=\mathrm{H}^{1}((\mathcal{O}(-1) \oplus \mathcal{O}_{L}(-1)) \otimes \mathcal{I}_{Z_{2}} \otimes \mathcal{O}(5))=\mathrm{H}^{1}(\mathcal{I}_{Z_{2}}(4)) \oplus \mathrm{H}^{1}(\mathcal{I}_{Z_{2}} \otimes \mathcal{O}_{L}(4))=0,
$$ and when  $L \not \subset P$, we have
$\mathrm{Ext}^{2}(B,A)=\mathrm{H}^{1}(\mathcal{I}_{p} \otimes \mathcal{I}_{Z_{2}}(5))=0.$

Finally, for $\mathrm{Ext}^2(B,B)$, applying Lemma \ref{exttrihuy}  to $j=\iota_{P}\colon P \rightarrow \mathbb{P}^{3}$ and $\mathcal{F}:=\mathcal{I}_{Z_{2}}(1)$, we get $\mathcal{I}_{Z_{2}}[1] \rightarrow \iota^{*}_{P}(\iota_{P_{*}}\mathcal{I}_{Z_{2}}(1))\rightarrow \mathcal{I}_{Z_{2}}(1),$
 and so applying $\mathrm{\mathrm{Hom}}(\mathcal{I}_{Z_{2}},-)$ to this exact triangle gives
$$\mathrm{Ext}^{2}(B,B)=\mathrm{Ext}^{2}(\iota_{P_{*}}(\mathcal{I}_{Z_{2}})^{\vee}(-5),\iota_{P_{*}}(\mathcal{I}_{Z_{2}})^{\vee}(-5))mathcal{I}_{Z_{2}},\iota^{!}_{P}(\iota_{P_{*}}\mathcal{I}_{Z_{2}}))$$$$
=\mathrm{Ext}^{1}(\mathcal{I}_{Z_{2}},\iota^{*}_{P}(\iota_{P_{*}}\mathcal{I}_{Z_{2}}(1))) = \mathrm{Ext}^{1}(\mathcal{I}_{Z_{2}},\mathcal{I}_{Z_{2}}(1))=\mathbb{C}^{4}.$$

\end{proof}

Let $\mathfrak{Fl}_2$ be the space  parametrising flags $Z_{2} \subset P \subset \mathbb{P}^{3}$ where $P$ is a plane and $Z_{2}$ a zero-dimensional subscheme of length 2. 

\begin{cor} \label{Cor: M+}
 The moduli space  $\mathcal{M}_{\sigma_{+}}(v)$ for the relevant chamber consists of two irreducible components: one is  $\widetilde {\mathcal{M}_{\sigma_{-}}(v)}$, birational to  $\mathcal{M}_{\sigma_{-}}(v)$, and a new component,  $\mathcal{M}'$, which is  a $\mathbb{P}^{17}$-bundle over $G(2,4) \times \mathfrak{Fl}_2 $.
\end{cor}

\begin{proof}Any object in $\mathcal{M}'$ fits into a short exact sequence $\mathcal{I}_{L}(-1)   \rightarrow E \rightarrow   \iota_{P_{*}}(\mathcal{I}_{Z_{2}})^{\vee}(-5)$.
Now we just notice that from Lemma ~\ref{ext1purple1}, we have $\mathrm{Ext}^{1}(\iota_{P_{*}}(\mathcal{I}_{Z_{2}})^{\vee}(-5),\mathcal{I}_{L}(-1))= \mathbb{C}^{18}$. This gives the description of $\mathcal{M}'$ as $\mathbb{P}^{17}$-bundle. As the set of objects in $\mathcal{M}_{\sigma_{-}}(v)$ that are also $\sigma_{+}$-stable is open, its closure $\widetilde{\mathcal{M}_{\sigma_{-}}(v)}$ is birational to $\mathcal{M}_{\sigma_{-}}(v)$ (we notice that by Proposition \ref{N2} and Lemma \ref{U-+} below, the whole $\mathcal{M}_{\sigma_{-}}(v)$ is not destabilized by crossing the wall). As $\mathrm{dim} \mathcal{M}'=28 > 24=\mathrm{dim}  \widetilde{\mathcal{M}_{\sigma_{-}}(v)}$, the locus $\mathcal{M}'$ is its own irreducible component. 
\end{proof}

Before the next subsection, we have the following Lemma that follows from \cite{AHH}:

\begin{lem} \label{Lem: wall-moduli}
For a stability condition $\sigma_{0}$ on the wall $\mathcal{W}$, the moduli space $\mathcal{M}_{\sigma_{0}}(v)$ exists.
\end{lem}

\subsection{Intersection in $\mathcal{M}_{\sigma_{+}}(v)$  }
To describe the intersection of $\widetilde{\mathcal{M}_{\sigma_{-}}(v)}$ with the new component $\mathcal{M}'$ in $\mathcal{M}_{\sigma_{+}}(v)$, let  $\mathcal{U}_{-,+}$ be the destabilizing locus in $\mathcal{M}_{\sigma_{-}}(v)$. More precisely, we have

$$\mathcal{U}_{-,+}= \{E\colon \iota_{P_{*}}(\mathcal{I}_{Z_{2}})^{\vee}(-5) \rightarrow E \rightarrow \mathcal{I}_{L}(-1) \},$$
and
  $$\mathcal{M}'= \{E\colon\mathcal{I}_{L}(-1) \rightarrow E \rightarrow \iota_{P_{*}}(\mathcal{I}_{Z_{2}})^{\vee}(-5)  \}.$$

 \begin{lem} \label{U-+}
The locus $\mathcal{U}_{-,+}$ is of dimension $10$; notice that it contains the exceptional locus of  $\phi\colon \mathcal{M}_{\sigma_{-}}(v) \rightarrow \mathcal{M}_{\sigma_{0}}(v)$ of dimension 8 which is a $\mathbb{P}^{1}$-bundle over its 7-dimensional image under $\phi$. 
\end{lem}

\begin{proof}
From Lemma ~\ref{ext1purple1}, we have 

\begin{equation*}
    \mathrm{Ext}^{1}(\mathcal{I}_{L}(-1),\iota_{P_{*}}(\mathcal{I}_{Z_{2}})^{\vee}(-5))=
    \begin{cases}
    0, & \langle Z_2 \rangle \cap L = \emptyset\\
       \mathbb{C}, & \text{ $\langle Z_2 \rangle \cap L \neq \emptyset$  but $\langle Z_2 \rangle \neq L$}\\
     
      \mathbb{C}^{2}, &  \langle Z_2 \rangle =  L

    \end{cases}
  \end{equation*}

  Considering the cases where $\mathrm{Ext}^{1}(A,B)\neq 0$, the generic case
       $L \not\subset P$ and $ p\cup Z_{2}$ colinear is given by a  10-dimensional stratum ($\mathbb{P}^{0}$-bundle over the parameter space of the configuration).
       
       When $\mathrm{Ext}^{1}(A,B)=\mathbb{C}^{2}$, the parameter space of all the points $A,B$ is given by the configurations $ Z_{2} \subset L \subset P$  which gives a 8-dimensional stratum ($\mathbb{P}^{1}$-bundle over the parameter space of the configuration). By the Bayer-Macr\`{i} positivity Lemma (\cite{bm14}), $\phi $ contracts exactly the $\mathbb{P}^1$ coming from  $\mathrm{Ext}^{1}(A,B)=\mathbb{C}^{2}$. 
\end{proof}

  \begin{rmk} \label{defE}
Recall that $A=\mathcal{I}_{L}(-1)$, $B=\iota_{P_{*}}(\mathcal{I}_{Z_{2}})^{\vee}(-5)$. Let 
 $E \in \mathcal{M}'$ fitting into a short exact sequence $A \hookrightarrow E \twoheadrightarrow B$. Deformations of $E$ within the exceptional locus corresponds to maps $E \rightarrow E[1]$ fitting into a map of exact triangles

 \begin{center}
    
\begin{tikzcd}[column sep=
scriptsize
]

A   \arrow[r, "a"] \arrow[d]  & E \arrow[r] \arrow[d, "f"] & B  \arrow[d]\\
A[1] \arrow[r]  & E[1] \arrow[r, "b"]  & B[1] 

\end{tikzcd}
\end{center}
 This is equivalent to $b \circ f \circ a=0$. In particular, if there are deformations of $E$ that do not remain in the exceptional locus, then $\mathrm{Ext}^{1}(A,B) \neq 0$.
  
  \end{rmk}

\subsubsection{Surjectivity}
  The goal of this subsection is to show that 
  $$\delta\colon \mathrm{Ext}^{1}(B,A) \rightarrow \mathrm{Hom}(\mathrm{Ext}^{1}(A,B),  \mathrm{Ext}^{2}(B,B))$$ (which is induced by the natural map  $\delta' \colon \mathrm{Ext}^{1}(B,A) \otimes \mathrm{Ext}^{1}(A,B)  \rightarrow  \mathrm{Ext}^{2}(B,B)$) is surjective. This will be needed in the next subsection to describe the desired intersection.

Before proving the surjectivity, we begin with the following Lemma:
\begin{lem} \label{surj1}
 If $\delta$ is surjective when $\mathrm{Ext}^{1}(A,B)=\mathbb{C}^2$, then it would be surjective when $\mathrm{Ext}^{1}(A,B)=\mathbb{C}$ as well.
\end{lem}
\begin{proof}
As the domain $\mathrm{Ext}^{1}(B,A)=\mathbb{C}^{18}$ is constant for all the configurations, surjectivity of $\delta$ is an open condition. Since all $PGL(3)$-orbits of configurations of $Z_{2} \subset P$ and $L$ contain the case $Z_{2} \subset L \subset P$ in its closure, for which $\mathrm{Ext}^{1}(A,B)=\mathbb{C}^2$, this proves the claim.
\end{proof}
Hence we only need to prove the surjectivity when  $\mathrm{Ext}^{1}(A,B)=\mathbb{C}^2$:
\begin{lem} \label{surj2}
 Assume that $\mathrm{Ext}^{1}(A,B)=\mathbb{C}^{2}$. Then $\delta$ is surjective.
\end{lem}

\begin{proof}

First, we simplify the map $$\delta'\colon \mathrm{Ext}^1(\iota_{P_{*}}\mathcal{I}^{\vee}_{Z_{2}}(-5),\mathcal{I}_{L}(-1))\otimes \mathrm{Ext}^1(\mathcal{I}_{L}(-1),\iota_{P_{*}}\mathcal{I}^{\vee}_{Z_{2}}(-5)) \rightarrow \mathrm{Ext}^2(\iota_{P_{*}}\mathcal{I}^{\vee}_{Z_{2}}(-5), \iota_{P_{*}}\mathcal{I}^{\vee}_{Z_{2}}(-5)).$$
\begin{description}

\item[$\boldsymbol{Claim \: 1}$] $\delta'$ is isomorphic to  $$\mathrm{\mathrm{Hom}}(\mathcal{O}_{P}(2) \oplus \mathcal{O}_{L}(3)[-1], \mathcal{I}_{Z_{2}}(6))\otimes \mathrm{Ext}^1(\mathcal{I}_{Z_{2}}(5),\mathcal{O}_{P}(2) \oplus \mathcal{O}_{L}(3)[-1])  \rightarrow$$
$$ \rightarrow \mathrm{Ext}^1(\mathcal{I}_{Z_{2}}(5), \mathcal{I}_{Z_{2}} (6)).$$

\underline{Proof of Claim 1:} Applying $\iota^{*}_{P}$  to $\delta'$ gives the following map:
$$ \iota^{*}_{P}\delta'\colon \mathrm{Ext}^1(\iota^{*}_{P}\iota_{P_{*}}\mathcal{I}^{\vee}_{Z_{2}}(-5),\iota^{*}_{P}\mathcal{I}_{L}(-1))\otimes \mathrm{Ext}^1(\iota^{*}_{P}\mathcal{I}_{L}(-1),\iota^{*}_{P}\iota_{P_{*}}\mathcal{I}^{\vee}_{Z_{2}}(-5)) \rightarrow $$$$\rightarrow \mathrm{Ext}^2(\iota^{*}_{P}\iota_{P_{*}}\mathcal{I}^{\vee}_{Z_{2}}(-5),\iota^{*}_{P}\iota_{P_{*}}\mathcal{I}^{\vee}_{Z_{2}}(-5)).$$

       We note that $\iota^{*}_{P}\mathcal{I}_{L}(-1)=\mathcal{I}_{L/P}(-1) \oplus \mathcal{O}_{L}(-2)$. Using this and Lemma \ref{lemhuy} implies that the map $ \iota^{*}_{P}\delta'$ projects to:
        $$ \mathrm{\mathrm{Hom}}((\mathcal{I}^{\vee}_{Z_{2}}(-6),\mathcal{O}_{P}(-2) \oplus \mathcal{O}_{L}(-2) )\otimes \mathrm{Ext}^1(\mathcal{O}_{P}(-2) \oplus \mathcal{O}_{L}(-2),\mathcal{I}^{\vee}_{Z_{2}}(-5)) \rightarrow $$$$\rightarrow \mathrm{Ext}^1( \mathcal{I}^{\vee}_{Z_{2}}(-6),\mathcal{I}^{\vee}_{Z_{2}}(-5)),$$
        which can be written as (noting that $\mathcal{O}^{\vee}_{L}= (\mathcal{O}_{P}(-1) \rightarrow \mathcal{O}_{P})^{\vee}$ is isomorphic to $(\mathcal{O}_{P} \rightarrow \mathcal{O}_{P}(1) )=\mathcal{O}_{L}(1)[-1]$, using the degree change):
        $$ \mathrm{\mathrm{Hom}}(\mathcal{O}_{P}(2) \oplus \mathcal{O}_{L}(3)[-1], \mathcal{I}_{Z_{2}}(6))\otimes \mathrm{Ext}^1(\mathcal{I}_{Z_{2}}(5),\mathcal{O}_{P}(2) \oplus \mathcal{O}_{L}(3)[-1]) \rightarrow $$$$\rightarrow \mathrm{Ext}^1(\mathcal{I}_{Z_{2}}(5), \mathcal{I}_{Z_{2}} (6)),$$
        and so Claim 1 is proved. 
        
        Now, using Serre duality for $\mathrm{Ext}^1(\mathcal{I}_{Z_{2}}(5),\mathcal{O}_{P}(2) \oplus \mathcal{O}_{L}(3)[-1])$ we have
    $$\mathrm{Ext}^1(\mathcal{I}_{Z_{2}}(5),\mathcal{O}_{P}(2) )=\mathrm{Ext}^1(\mathcal{O}_{P}(2),\mathcal{I}_{Z_{2}}(2) )^{\vee}=\mathbb{C},$$
    and 
    $$\mathrm{Ext}^1(\mathcal{I}_{Z_{2}}(5), \mathcal{O}_{L}(3)[-1])=\mathrm{\mathrm{Hom}}(\mathcal{I}_{Z_{2}}(5), \mathcal{O}_{L}(3))=\mathbb{C}.$$
    
    Also for $\mathrm{\mathrm{Hom}}(\mathcal{O}_{P}(2) \oplus \mathcal{O}_{L}(3)[-1], \mathcal{I}_{Z_{2}}(6))$, component-wise we have (after applying $\mathrm{R}\mathrm{\mathrm{Hom}}(-,\mathcal{I}_{Z_{2}}(3))$ to $\mathcal{O}_{P}(-1) \rightarrow \mathcal{O}_{P} \rightarrow \mathcal{O}_{L}$):
        $$\mathrm{\mathrm{Hom}}(\mathcal{O}_{P}(2), \mathcal{I}_{Z_{2}}(6))=\mathrm{H}^0(\mathcal{I}_{Z_{2}}(4))=\mathbb{C}^{13},$$
        and 
        $$\mathrm{\mathrm{Hom}}(\mathcal{O}_{L}(3)[-1], \mathcal{I}_{Z_{2}}(6))=\mathrm{Ext}^1(\mathcal{O}_{L}, \mathcal{I}_{Z_{2}}(3))= \mathbb{C}^5.$$
        As $\mathrm{Ext}^1(\mathcal{I}_{Z_{2}}(5), \mathcal{I}_{Z_{2}} (6))=\mathbb{C}^{4}$, the above map is the same as the following two maps:
        $$\mathbb{C}^{13} \otimes \mathbb{C}=\mathbb{C}^{13}  \rightarrow \mathbb{C}^{4}, \quad  \text{and} \quad \mathbb{C}^{5}  \otimes \mathbb{C}=\mathbb{C}^{5} \rightarrow \mathbb{C}^{4},$$
   i.e., we have the following diagram

\begin{center}
    
\begin{tikzcd}[column sep=
scriptsize
]
& \mathcal{O}_{P}(-3)[1]   \arrow["\mathbb{C}^{13}",dr ]\\
\mathcal{I}_{Z_{2}} \arrow["\mathbb{C}",ur ]  \arrow["\mathbb{C}",dr ] \arrow[rr, rightarrow, "\mathbb{C}^{4}"]{}
& & \mathcal{I}_{Z_{2}}(1)[1]. \\
& \mathcal{O}_{L}(-2) \arrow[ur, "\mathbb{C}^{5}"]
\end{tikzcd}
\end{center}

\item[$\boldsymbol{Claim \: 2}$] $\mathbb{C} \otimes \mathbb{C}^{13}\rightarrow \mathbb{C}^{4}$ is surjective.

\underline{Proof of Claim 2:} For $\mathbb{C} \otimes \mathbb{C}^{13}  \rightarrow \mathbb{C}^{4}$ or $\mathrm{Ext}^1(\mathcal{I}_{Z_{2}},\mathcal{O}_{P}(-3)) \otimes \mathrm{\mathrm{Hom}}(\mathcal{O}_{P}(-3),\mathcal{I}_{Z_{2}}(1)) \\ \rightarrow \mathrm{Ext}^1(\mathcal{I}_{Z_{2}},\mathcal{I}_{Z_{2}}(1))$, we consider $E \in \mathrm{Ext}^1(\mathcal{I}_{Z_{2}},\mathcal{O}_{P}(-3)) $ and apply $\mathrm{\mathrm{RHom}}(-,\mathcal{I}_{Z_{2}}(1) )$ to $\mathcal{O}_{P}(-3) \rightarrow E \rightarrow \mathcal{I}_{Z_{2}} $, and so we have

\begin{center}

\begin{tikzpicture}[descr/.style={fill=white,inner sep=1.5pt}]
        \matrix (m) [
            matrix of math nodes,
            row sep=1em,
            column sep=2.5em,
            text height=1.5ex, text depth=0.25ex
        ]
        { 0  &\mathrm{\mathrm{Hom}}(\mathcal{I}_{Z_{2}},\mathcal{I}_{Z_{2}}(1)) &\mathrm{\mathrm{Hom}}(E,\mathcal{I}_{Z_{2}}(1)) & \mathrm{\mathrm{Hom}}(\mathcal{O}_{P}(-3),\mathcal{I}_{Z_{2}}(1))\\
             & \mathrm{Ext}^1(\mathcal{I}_{Z_{2}},\mathcal{I}_{Z_{2}}(1)) & \mathrm{Ext}^1(E,\mathcal{I}_{Z_{2}}(1)) & 0\\
        };

        \path[overlay,->, font=\scriptsize,>=latex]
        (m-1-1) edge (m-1-2)
        (m-1-2) edge (m-1-3)
        (m-1-3) edge (m-1-4)
        (m-1-4) edge[out=355,in=175] node[descr,yshift=0.3ex] {} (m-2-2)
        (m-2-2) edge (m-2-3)
        (m-2-3) edge (m-2-4)
      ;
\end{tikzpicture}
\end{center}

   For computing $\mathrm{Ext}^1(E,\mathcal{I}_{Z_{2}}(1))$, as $\mathrm{Ext}^1( \mathcal{O}_{P}(-1), \mathcal{O}_{P}(-3))=0$ from the top and the middle row of diagram below, we get $ \mathcal{O}_{P}(-1) \hookrightarrow E$, and so we can complete the diagram:

   \begin{center}
    
\begin{tikzcd}[column sep=
scriptsize
]
 0 \arrow[r, hook, ""]  \arrow[d, hook ] & \OO_P(-1) \arrow[r, dash, "\cong"] \arrow[d, hook ]&  \mathcal{O}_{P}(-1)  \arrow[dr, dashrightarrow, "0" ]   \arrow[d, hookrightarrow, "" ]\\
\mathcal{O}_{P}(-3)   \arrow[r,  ""]  \arrow[d, dash, "\cong" ]  &E\arrow[r, ""]  \arrow[d, two heads ] 
&  \mathcal{I}_{Z_{2}} \arrow[r,  ""]  \arrow[d, two heads ] & \mathcal{O}_{P}(-3) [1] \\
\mathcal{O}_{P}(-3) \arrow[r, dashed,  hook,  ""]  & Q \arrow[r,dashed,  two heads, ""] &
 \mathcal{O}_{L}(-2),
\end{tikzcd}
\end{center} 
  where $Q$ is the quotient of $\mathcal{O}_{P}(-1) \rightarrow E$. Therefore, the bottom row implies $Q=\mathcal{O}_{P}(-2)$. But as $\mathrm{Ext}^1(\mathcal{O}_{P}(-2),\mathcal{O}_{P}(-1))=0$, we are left with the trivial extension $E=\mathcal{O}_{P}(-1) \oplus \mathcal{O}_{P}(-2)$. Therefore we have
   $$\mathrm{Ext}^1(E,\mathcal{I}_{Z_{2}}(1))=\mathrm{Ext}^1(\mathcal{O}_{P}(-1) \oplus \mathcal{O}_{P}(-2),\mathcal{I}_{Z_{2}}(1))=\mathrm{H}^1(\mathcal{I}_{Z_{2}}(2))\oplus \mathrm{H}^1(\mathcal{I}_{Z_{2}}(3))=0,$$
   i.e., from the above long exact sequence we have the surjection
   $$\mathrm{Hom}(\mathcal{O}_{P}(-3),\mathcal{I}_{Z_{2}}(1))=\C^{13}\twoheadrightarrow \mathrm{Ext}^1(\mathcal{I}_{Z_{2}},\mathcal{I}_{Z_{2}}(1)=\C^4,$$

  \item[$\boldsymbol{Claim \: 3}$]$\mathbb{C} \otimes \mathbb{C}^{5}  \rightarrow \mathbb{C}^{4}$ is surjective.

   \underline{Proof of Claim 3:} Applying $R\mathrm{\mathrm{Hom}}(-,\mathcal{I}_{Z_{2}}(1))$ on $\mathcal{O}(-1)=\mathcal{I}_{L} \hookrightarrow \mathcal{I}_{Z_{2}} \twoheadrightarrow \mathcal{I}_{Z_{2}/L}=\mathcal{O}_{L}(-2)$ gives the surjection
 $$\mathrm{Ext}^1(\mathcal{O}_{L}(-2),\mathcal{I}_{Z_{2}}(1))=\C^5 \twoheadrightarrow \mathrm{Ext}^1(\mathcal{I}_{Z_{2}},\mathcal{I}_{Z_{2}}(1))=\C^4,$$
 \end{description}
  Therefore, $\delta$ is surjective when $\mathrm{Ext}^1(A,B)=\mathbb{C}^2$.
  
  \end{proof}

Therefore Lemma \ref{surj1} implies that $\delta$ is always surjective.

  \subsubsection{Intersection of the components} In this subsection, we  describe the intersection $\widetilde {\mathcal{M}_{\sigma_{-}}(v)} \\ \cap \mathcal{M}'$. So far, we have proved that the map $\delta\colon \mathrm{Ext}^1(B,A) \rightarrow \mathrm{\mathrm{Hom}}(\mathrm{Ext}^1(A,B), \mathrm{Ext}^2(B,B))$ is surjective. 
Consider the map $\zeta\colon \mathrm{Ext}^{1}(A,B) \rightarrow \mathrm{Ext}^{2}(B,B)$, which is defined by composition with a class in $\mathrm{Ext}^{1}(B,A)$.

   \begin{lem} \label{ext1E}
    For any $E \in \mathcal{M}'$, we have 
    $$\mathrm{ext}^1(E,E)=28+dim (\mathrm{ker} (\zeta)).$$
    
    \end{lem}

   \begin{proof}
   We consider the long exact $\mathrm{Hom}$-sequence associated to the short exact sequence $A \rightarrow E \rightarrow B$. Using Lemma \ref{ext1purple1},  we get

  \begin{center}
    \begin{tikzpicture}[>=angle 90]
\matrix(a)[matrix of math nodes,
row sep=2.5em, column sep=0.5em,
text height=1.5ex, text depth=0.25ex]
{&
& \mathrm{Ext}^{2}(A,A)=0 & 0 &\mathrm{Ext}^{2}(B,A)=0&&&\\&
\mathbb{C}^{4} \supset \mathrm{Im}(\zeta)& \mathrm{Ext}^{1}(A,B) & \mathrm{Ext}^{1}(E,B) & \mathbb{C}^{7}& 0& &\\&
&&\mathrm{Ext}^{1}(E,E)&&&&
\\&
0&\mathbb{C}^{4}&\mathrm{Ext}^{1}(E,A)&\mathbb{C}^{18}&\mathbb{C}&0&\\&
& 0&\mathrm{Hom}(E,B)=\mathbb{C}&\mathbb{C}&0&&\\
&0&\mathrm{Hom}(A,E)=\mathbb{C}&\mathrm{Hom}(E,E)=\mathbb{C}&0&&&&\\
&&&\mathrm{Hom}(E,A)=0.\\
};

\path [<- ](a-1-4) edge  (a-1-5);

\path [<- ](a-1-3) edge  (a-1-4);

\path [<- ](a-2-2) edge  (a-2-3);

\path [<- ](a-2-3) edge  (a-2-4);
\path [<- ](a-2-4) edge  (a-2-5);
\path [<- ](a-2-5) edge  (a-2-6);

\path[<-](a-6-4) edge (a-7-4);
\path[<-](a-5-4) edge (a-6-4);
\path[<-](a-4-4) edge (a-5-4);
\path[<-](a-3-4) edge (a-4-4);
\path[<-](a-2-4) edge (a-3-4);
\path[<-](a-1-4) edge (a-2-4);

\path[<-](a-6-4) edge (a-6-5);
\path[<-](a-6-3) edge (a-6-4);
\path[<-](a-6-2) edge (a-6-3);

\path[<-](a-5-5) edge (a-5-6);
\path[<-](a-5-4) edge (a-5-5);
\path[<-](a-5-3) edge (a-5-4);

\path[<-](a-4-6) edge (a-4-7);
\path[<-](a-4-5) edge (a-4-6);
\path[<-](a-4-4) edge (a-4-5);
\path[<-](a-4-3) edge (a-4-4);
\path[<-](a-4-2) edge (a-4-3);

\end{tikzpicture}
\end{center}
  The second row gives $\mathrm{ext}^{1}(E,B)=7+\mathrm{dim}(\mathrm{ker}(\zeta))$. On the other hand, the fourth row implies $\mathrm{ext}^{1}(E,A)=21.$ Therefore, the column of $\mathrm{Ext}^{1}(E,E)$ gives
  $[7+\mathrm{dim}(\mathrm{ker}(\zeta))]-\mathrm{ext}^{1}(E,E)+21=0,$ which implies the claim.
  
  \end{proof}

As $\mathrm{dim}(\mathcal{M}')=28$, the locus where $\mathrm{ext}^{1}(E,E) > 28$ is the singular locus in $\mathcal{M}'$, which we will now describe in more detail in Proposition \ref{P}. 
\begin{prop} \label{P} 
 Let $\mathcal{R}$ be the locus in $\mathcal{M}'$ where  $\mathcal{M}_{\sigma_{+}}(v)$ is singular. Let $\widetilde{\psi}\colon \mathcal{M}_{\sigma_{+}}(v) \rightarrow \mathcal{M}_{\sigma_{0}}(v)$. Then the restriction
$\widetilde{\psi}|_\mathcal{R}$ is generically a $\mathbb{P}^{13}$-bundle over a 10-dimensional base, degenerating to a bundle over a 7-dimensional base whose fibers are a 14-dimensional cone with $\mathbb{P}^{9}$ as vertex.

\end{prop}
\begin{proof}

 Given $A$ and $B$, the singular locus is a subset of $\mathbb{P}\mathrm{Ext}^1(B, A)$-bundle, such that $\mathrm{ker}(\zeta)$ is non-zero (by Lemma \ref{ext1E}).
In other words, we want to find all extension classes $\vartheta \in \mathrm{Ext}^1(B, A)$ such that the induced map $\zeta\colon \mathrm{Ext}^{1}(A,B) \rightarrow \mathrm{Ext}^{2}(B,B)$ has non-trivial kernel. Having a non-trivial kernel means that  $\zeta$ cannot have full rank and hence the locus is given by the projectivization of those $\vartheta$ for which $\zeta=\delta(\vartheta)$ drops rank.

Now there are two non-zero cases of $\mathrm{Ext}^{1}(A,B)$:
\begin{description}

\item [$\boldsymbol{Case (I) \quad }\mathrm{Ext}^1(A,B)=\mathbb{C}$]
  
 In this case, dropping rank means $\zeta=\delta(\vartheta)=0$. Also as in this case we have $\delta\colon \mathbb{C}^{18} \rightarrow \mathbb{C}^{4}$, and by Lemmas \ref{surj1} and \ref{surj2}, $\delta$ is surjective, we will have $\mathrm{ker}(\delta)= \mathbb{C}^{14}$. Therefore in this fiber we get  $\mathbb{P}\mathrm{ker}(\delta)=\mathbb{P}^{13}$.
 
 The base depends on the configurations of $Z_{2}, L$ and $P$: The generic case is when \text{$L \not\subset P$ and $p \cup Z_{2}$ colinear}, which gives a 10-dimensional base (3+4+3 in which 3 is for the plane, 4 for two points in the plane, and 3 for the spatial line with one condition). Therefore, we have a $\mathbb{P}^{13}$-bundle over a 10-dimensional locus, and so the objects $E$ with $\zeta$ non-injective correspond to a 23-dimensional space, in this case.
 
$$$$

\item[$\boldsymbol{Case (II)} \quad \mathrm{Ext}^1(A,B)=\mathbb{C}^2$]

   In this case, dropping rank means we have to consider
$$\mathcal{V}=\mathbb{P}(\delta^{-1} \text{(rank $\leq1$ matrices in }  \mathrm{Hom}(\mathrm{Ext}^1(A, B), \mathrm{Ext}^2(B, B)) \cong \mathrm{Hom}(\mathbb{C}^{2},\mathbb{C}^{4}))).$$

 Lemma \ref{ext1purple1} says that we are  in the  case ${Z_{2} \subset L \subset P}$, so the base is a 7-dimensional locus (3  for the plane, and 4 for two points in the plane). Now we want to find 
 the fiber locus. There are two  non-injective possibilities for $\mathrm{ker}(\zeta)$:
 
(1) $\mathrm{ker}(\zeta)= \mathbb{C}$.  
In this case,  as a cone over \text{$\mathbb{P}(Im \delta \cap$ (rank $1$ matrices))}, the corresponding part of $\widetilde {\mathcal{M}_{\sigma_{-}}(v)} \cap \mathbb{P}(\mathrm{Ext}^1(B,A))$ as a fiber over the base, is $\mathcal{V}$ which is a subset of $\mathbb{P}(\mathrm{Ext}^1(B,A))=\mathbb{P}^{17}$. 
Consider the map $\mathbb{C}^{2} \rightarrow \mathbb{C}^{4}$ and note that $\mathrm{ker}(\zeta)= \mathbb{C}$ is equivalent to choose 2 linearly dependent vectors in $\mathbb{C}^{4}$, which is a condition of codimension $3$. Therefore the desired preimage is of codimension 3 as well. Therefore, the intersection $\widetilde {\mathcal{M}_{\sigma_{-}}(v)} \cap \mathbb{P}(\mathrm{Ext}^1(B,A))$ is a (projective) 14-dimensional subset of $\mathbb{P}^{17}$. As the dimension of the fiber increases in this case, the corresponding fiber is the degeneration of the $\mathbb{P}^{13}$.

 (2) $\mathrm{ker}(\zeta)= \mathbb{C}^{2}$.
In this case, the codomain of $\delta$ is the space of 4 by 2 matrices, so we have $\mathrm{dim}(\mathrm{ker}(\delta))=10$ as $\delta$ is surjective by Lemma \ref{surj2}, which means that $\mathbb{P}(\mathrm{ker} \delta)$ is 9-dimensional as a subset of $\widetilde{\mathcal{M}_{\sigma_{-}}(v)} \cap \mathbb{P}(\mathrm{Ext}^1(B,A))$. As $\zeta$ is trivial in this case, this locus in the preimage of $\delta$ is the singularity locus (or the vertex) of the 14-dimensional cone.

Therefore as the degeneration of the $\mathbb{P}^{13}$-bundle, we have a 14-dimensional fiber cone over a 7 dimensional base (with vertex a $\mathbb{P}^{9}$-bundle over the 7-dimensional base).
 \end{description}

\end{proof}

Because $\mathcal{M}'$ is smooth of dimension 28, the intersection  $\widetilde{\mathcal{M}_{\sigma_{-}}(v)} \cap \mathcal{M}'$ is contained in the singular locus $\mathcal{R}$. We now want to prove the converse $\mathcal{R} \subset \widetilde{\mathcal{M}_{\sigma_{-}}(v)} \cap \mathcal{M}'$. We first consider the following subsets of $\widetilde{\psi}(\mathcal{R})$:
\begin{itemize}
 \item Let $U_{1} \subset G(2,4) \times \mathfrak{Fl}_2$ be the locus where $L \not \subset P$ and $p=L \cap P$ is colinear with $Z_2$, but $p$ is disjoint from $Z_2$, and $Z_2$ consists of two distinct points.
 
 \item Let  $U_{2} \subset G(2,4) \times \mathfrak{Fl}_2$ be the locus where $Z_2$ consists of two distinct points, and L is the line spanned by $Z_2$.
\end{itemize}

We observe that $U_1$ and $U_2$ are open and dense in the loci of $ G(2,4) \times \mathfrak{Fl}_{2}$ where $\mathrm{Ext}^1(A,B)\\=\mathbb{C}$ and 
$\mathrm{Ext}^1(A,B)=\mathbb{C}^2$, respectively. Hence $\mathcal{R}_1:=  (\widetilde{\psi}|_\mathcal{R})^{-1} (U_1)$ and  $\mathcal{R}_2:=  (\widetilde{\psi}|_\mathcal{R})^{-1} (U_2)$ are dense in the loci where $\widetilde{\psi}|_\mathcal{R}$ is a $\mathbb{P}^{13}$-bundle or a 14-dimensional cones bundle, respectively; as the intersection $\widetilde{\mathcal{M}_{\sigma_{-}}(v)} \cap \mathcal{M}'$  is closed, it is therefore enough to show that $\mathcal{R}_1$ and $\mathcal{R}_2$ are contained in $\widetilde{\mathcal{M}_{\sigma_{-}}(v)} \cap \mathcal{M}'$.

  \begin{lem} \label{{H}^{i}(E)}

 Let  $E \in \mathcal{R}_1$ (i.e., $E$ is the general element of the $\mathbb{P}^{13}$-bundle). We have
  \begin{itemize}
      \item $\HH^{0}(E)=\mathcal{I}_{C_{5} \cup L}$, where $C_{5} \subset P$ is a plane quintic containing $L \cap P$ and having a node at each point of $Z_2$, such that $L \cap P \sub \langle Z_2 \rangle $,  and
      \item  $\HH^{1}(E)=\mathcal{O}_{Z_{2}}$.
  \end{itemize}

  \end{lem}
  \begin{proof}
  
  We want to describe the $\mathbb{P}^{13}$-bundle in Proposition \ref{P} more precisely. Imitating the beginning of the proof of Lemma \ref{surj2} for the generic case $\langle Z_2 \rangle \cap L \neq \emptyset$ and  $\langle Z_2 \rangle \neq L$, and applying $\iota^{*}_{P}$  to $\delta'$ gives the following map:
$$\iota^{*}_{P}\delta'\colon \mathrm{Ext}^1(\iota^{*}_{P}\iota_{P_{*}}(\mathcal{I}_{Z_{2}})^{\vee}(-5),\iota^{*}_{P}\mathcal{I}_{L}(-1))\otimes \mathrm{Ext}^1(\iota^{*}_{P}\mathcal{I}_{L}(-1),\iota^{*}_{P}\iota_{P_{*}}(\mathcal{I}_{Z_{2}})^{\vee}(-5)) \rightarrow $$$$\rightarrow \mathrm{Ext}^2(\iota^{*}_{P}\iota_{P_{*}}(\mathcal{I}_{Z_{2}})^{\vee}(-5),\iota^{*}_{P}\iota_{P_{*}}(\mathcal{I}_{Z_{2}})^{\vee}(-5)).$$

We notice that $\iota^{*}_{P}\mathcal{I}_{L}(-1)=\mathcal{I}_{p}(-1)$ in the generic case, and so using Lemma \ref{lemhuy} implies that the map $\iota^{*}_{P}\delta'$ projects to
 $$ \mathrm{\mathrm{Hom}}((\mathcal{I}_{Z_{2}})^{\vee}(-6),\mathcal{I}_{p}(-1))\otimes \mathrm{Ext}^1(\mathcal{I}_{p}(-1),(\mathcal{I}_{Z_{2}})^{\vee}(-5)) \rightarrow \mathrm{Ext}^1( (\mathcal{I}_{Z_{2}})^{\vee} (-6),(\mathcal{I}_{Z_{2}})^{\vee}(-5)),$$
        which can be written as :
        $$ \mathrm{\mathrm{Hom}}((\mathcal{I}_{p})^{\vee} (1), \mathcal{I}_{Z_{2}}(6))\otimes \mathrm{Ext}^1(\mathcal{I}_{Z_{2}}(5),(\mathcal{I}_{p})^{\vee} (1)) \rightarrow \mathrm{Ext}^1(\mathcal{I}_{Z_{2}}(5), \mathcal{I}_{Z_{2}} (6)).$$
        Applying Serre Duality we  have 
        $$ \mathrm{\mathrm{Hom}}((\mathcal{I}_{p})^{\vee} (1), \mathcal{I}_{Z_{2}}(6))\otimes \mathrm{Ext}^1((\mathcal{I}_{p})^{\vee} (1), \mathcal{I}_{Z_{2}}(2))^{\vee} \rightarrow \mathrm{Ext}^1(\mathcal{I}_{Z_{2}}(6), \mathcal{I}_{Z_{2}} (2))^{\vee}.$$
        Rearranging this  gives
 $$ \mathrm{\mathrm{Hom}}((\mathcal{I}_{p})^{\vee} (1), \mathcal{I}_{Z_{2}}(6))\otimes \mathrm{Ext}^1(\mathcal{I}_{Z_{2}}(6), \mathcal{I}_{Z_{2}} (2)) \rightarrow \mathrm{Ext}^1((\mathcal{I}_{p})^{\vee} (1), \mathcal{I}_{Z_{2}}(2)).$$
 
 Applying $\mathrm{Hom}(\mathcal{I}_{Z_{2}}(6), -)$ and $\mathrm{Hom}((\mathcal{I}_{p})^{\vee} (1), -)$ to $\mathcal{I}_{Z_{2}}(2) \hookrightarrow \mathcal{O}(2) \twoheadrightarrow \mathcal{O}_{Z_{2}} $ implies the map can be written as
  $$ \mathrm{\mathrm{Hom}}((\mathcal{I}_{p})^{\vee} (1), \mathcal{I}_{Z_{2}}(6))\otimes \mathrm{Hom}(\mathcal{I}_{Z_{2}}(6), \mathcal{O}_{Z_{2}} ) \rightarrow \mathrm{Hom}((\mathcal{I}_{p})^{\vee} (1), \mathcal{O}_{Z_{2}})/\mathrm{Im}(\mathrm{Hom}(\mathcal{O}(2), \mathcal{O}_{Z_{2}})).$$
 But as $\mathrm{Hom}(\mathcal{I}_{Z_{2}}(6), \mathcal{O}_{Z_{2}} ) =\mathrm{H}^{0}( \mathcal{I}_{Z_{2}}/ \mathcal{I}^{2}_{Z_{2}}(6))^{\vee}$, the map $\eta\colon \mathrm{Ext}^{1}(B,A) \rightarrow \mathrm{Ext}^{2}(B,B)$ (which is induced by the above map) can be written as 
 $$\mathrm{H}^{0}(\mathcal{I}_{Z_{2}} \otimes \mathcal{I}_{p} (5))  \rightarrow \mathrm{H}^{0}( \mathcal{I}_{Z_{2}}/ \mathcal{I}^{2}_{Z_{2}}(6)).$$
 Therefore $\mathrm{ker}(\eta)=\mathrm{H}^{0}(\mathcal{I}^{2}_{Z_{2}} \otimes \mathcal{I}_{p} (5))=\mathbb{C}^{14}$ (notice that we have $\mathrm{H}^{0}(\mathcal{O} (5))=\mathbb{C}^{21}$, $\mathrm{H}^{0}(\mathcal{O}^{2}_{Z_{2}} )=\mathbb{C}^{6}$, and $\mathrm{H}^{0}(\mathcal{O}_{p} )=\mathbb{C}$). This exactly gives the $\mathbb{P}^{13}$-bundle above.

  The cohomology long exact sequence of the defining short exact sequence of $E$ is
  
  \begin{center}

\begin{tikzpicture}[descr/.style={fill=white,inner sep=1.5pt}]
        \matrix (m) [
            matrix of math nodes,
            row sep=1em,
            column sep=2.5em,
            text height=1.5ex, text depth=0.25ex
        ]
        { 0  &\HH^{0}(\mathcal{I}_{L}(-1))=\mathcal{I}_{L}(-1) &\HH^{0}(E) & \HH^{0}(\iota_{P_{*}}(\mathcal{I}_{Z_{2}})^{\vee}(-5))=\mathcal{O}_{P}(-5)\\
         &\HH^{1}(\mathcal{I}_{L}(-1))=0 &\HH^{1}(E) & \HH^{1}(\iota_{P_{*}}(\mathcal{I}_{Z_{2}})^{\vee}(-5))=\mathcal{O}_{Z_{2}}\\
     &\HH^{2}(\mathcal{I}_{L}(-1))=0.\\
        };

        \path[overlay,->, font=\scriptsize,>=latex]
        (m-1-1) edge (m-1-2)
        (m-1-2) edge (m-1-3)
        (m-1-3) edge (m-1-4)
        (m-1-4) edge[out=355,in=175] node[descr,yshift=0.3ex] {} (m-2-2)
        (m-2-2) edge (m-2-3)
        (m-2-3) edge (m-2-4)
        (m-2-4) edge[out=355,in=175] node[descr,yshift=0.3ex] {} (m-3-2)
        ;
\end{tikzpicture}
\end{center}
The second row implies $\HH^{1}(E)=\mathcal{O}_{Z_{2}}$. The first row 
implies $\mathrm{ch}_{\leq 2}(\HH^{0}(E))=(1,0,-6)$. Therefore $\HH^{0}(E)$ is a sheaf with $c_{0}=1$ and $c_{1}=0$. Therefore, in order to show $\mathrm{H}^{0}(E)$ is an ideal sheaf of a curve, we need to show that it is torsion-free. We know that any subsheaf of $\mathcal{O}_{P}(-5)$ contains a subsheaf of the form $\mathcal{O}_{P}(-i)$ for some $-i \leq -6$. It is enough to show that such $\mathcal{O}_{P}(-i)$ does not lift to a subobject of $\HH^{0}(E)$. However, using the identification $\mathrm{Ext}^1(\mathcal{O}_{P}(-i), \mathcal{I}_{L}(-1))=\mathrm{Ext}^1(\mathcal{O}_{P}(-i), \iota^{*}_{P}\mathcal{I}_{L}(-1)(1)[-1])=\mathrm{H}^0(\mathcal{O}_{P}(-i) \otimes \mathcal{I}_{L})$, one can see that the composition $  \mathrm{Ext}^{1}(\iota_{P_{*}}(\mathcal{I}_{Z_{2}})^{\vee}(-5), \mathcal{I}_{L}(-1))   \rightarrow \mathrm{Ext}^{1}(\mathcal{O}_{P}(-5), \mathcal{I}_{L}(-1)) \rightarrow \mathrm{Ext}^{1}(\mathcal{O}_{P}(-i), \mathcal{I}_{L}(-1)) $ is injective. Hence $\mathcal{O}_{P}(-i)$ does not lift to a subsheaf of $\HH^{0}(E)$.

Therefore $\HH^{0}(E)$ is an ideal sheaf of curves. Hence the argument above implies that this ideal sheaf is of the form $\mathcal{I}_{C_{5} \cup L}$. The condition  $L \cap P \sub \langle Z_2 \rangle $ comes from Lemma \ref{ext1purple1}.
   
  \end{proof}

   By Lemma  \ref{{H}^{i}(E)}, we assume that $C_{5}$ is smooth outside $Z_{2}$. Let $C'$ be the normalization of $C_{5}$ which is a smooth curve of genus 4. Let $\mathcal{L} :=  \mathcal{O}_{C'}(1)$, a line bundle of degree 5 with at least 3 global sections. By Riemann-Roch we have $\mathrm{h}^{0}(\mathcal{L})-\mathrm{h}^{1}(\mathcal{L})=1+\mathrm{deg}(\mathcal{L})-\mathrm{genus}(\mathcal{L})=2$, and therefore $\mathrm{h}^{0}(\omega_{C'}\otimes\mathcal{L}^{\vee})=\mathrm{h}^{1}(\mathcal{L})\geq 1$ for canonical bundle $\omega_{C'}$, and so $\omega_{C'}\otimes\mathcal{L}^{\vee}$ has section. Thus, as $\omega_{C'}\otimes\mathcal{L}^{\vee}$ has degree one, it is of the form $\mathcal{O}_{C'}(x)$, for some point $x \in C'$. Therefore, $\mathcal{L}=\omega_{C'}(-x)$ and it has exactly 3 global sections.

  \begin{lem} \label{C'}
 $C'$ is not hyper-elliptic. 
  \end{lem}
  
  \begin{proof}
  By construction, since $\omega_{C'}(-x)=\mathcal{O}_{C'}(1)$ is globally generated (as the pull-back of $\mathcal{O}_{\mathbb{P}^3}(1)$), we have $\mathrm{H}^0(\omega_{C'}(-x)) =\mathbb{C}^3$. Therefore, $\mathrm{H}^0(\omega_{C'}(-x-y)) = \mathbb{C}^2$, for any $y \in C'$. This means that $C'$ cannot be hyper-elliptic.

  \end{proof}
  \begin{lem} \label{pointx}
  $x$ is the same as $\langle Z_2 \rangle \cap C_5$.
  \end{lem}
  \begin{proof}

By Lemma \ref{C'}, $C'$ is non-hyperelliptic; so embedding via $\mathrm{H}^{0}(\omega_{C'}(-x))$ is the same as the canonical embedding via $\mathrm{H}^{0}(\omega_{C'})$ followed by projecting from the point $x$. The nodal points $Z_{2}$ correspond to two trisecant lines $l, l'$ containing $x$.  Let $P_{0}$ be the plane spanned by $l,l'$. As $C'$ is a canonical genus four curve, thus a $(2, 3)$-complete intersection curve, the intersection $C' \cap P_{0}$ is also a $(2, 3)$ complete intersection. But that is only possible if it contains $x$ as a double point (the quadratic equation has to be exactly the equation defining $l \cup l'$, and the cubic equation has to vanish at $x$), and thus its projection must be colinear to $Z_{2}$, i.e., $x=p=P \cap L$.
  \end{proof}

   For deforming $C' \cup L$, let us look at the family $\mathcal{C}:= Bl_{ 0}(C' \times \mathbb{A}^{1})$. We have the projection $\mathfrak{q}\colon \mathcal{C} \rightarrow  \mathbb{A}^{1}$; its fibers $C_{t}=\mathfrak{q}^{-1}(t)$ are $C_t=C'$ for $t \neq 0$, and    $C_0=C' \cup L$.

  \begin{lem} \label{globally}
    Let $\omega_{\mathfrak{q}}$ be the relative canonical bundle, and $\mathcal{L}': =\omega_{\mathfrak{q}}(-2L)$. For all $t \in \mathbb{A}^{1}$, we have $\mathrm{H}^{0}(\mathcal{L}'|_{C_{t}})=\mathbb{C}^{4}$, and $\mathcal{L}'|_{C_{t}}$ is globally generated. 
  \end{lem}
  \begin{proof}

  For $t\neq 0$,  $\mathcal{L}'|_{C_{t}}= \omega_{C_t}$ implies the claim as $C_t=C'$ is canonical. For $C_0$, using Lemma \ref{pointx}, we have
  $\omega_{\mathfrak{q}}|_{C'}=\omega_{C'}(x)$, and so we observe
  $$\mathcal{L}'|_{C'}=\omega_{\mathfrak{q}}(-2L)|_{C'}=\omega_{C'}(-x)=\mathcal{L}, \quad \text{and} \quad \mathcal{L}'|_{L}=\omega_{\mathfrak{q}}(-2L)|_{L}=\omega_{\mathfrak{q}}|_{L}(2)=\mathcal{O}_{L}(1).$$
 From the exact sequence $\mathcal{O}_{L} \rightarrow \mathcal{L}'|_{C_{0}} \rightarrow \mathcal{L}'|_{C'}$ and noticing that $\mathrm{H}^{0}(\mathcal{O}_{L})=\mathbb{C}^1$ and $\mathrm{H}^{0}( \mathcal{L}'|_{C'})=\mathrm{H}^{0}( \mathcal{L})=\mathbb{C}^3$,  we have the result.

  \end{proof}

 We will show that any $E \in \mathcal{R}_1$ is a limit of an ideal sheaf of canonical genus four curves in $\widetilde{\mathcal{M}_{\sigma_{-}}(v)} \cap \mathcal{M}'$:
  \begin{cor} \label{secondcor}
   All $E$ in $\mathbb{P}^{13}=\mathbb{P}(\mathrm{H}^{0}(\mathcal{I}^{2}_{Z_{2}} \otimes \mathcal{I}_{p} (5)))$ are limits of objects in $\widetilde{\mathcal{M}_{\sigma_{-}}(v)} \cap \mathcal{M}'$. 
  
  \end{cor}

  \begin{proof}
    Recall that as in Lemma \ref{{H}^{i}(E)}, we can assume that $E$ is generic. By Lemma \ref{globally}, $\mathbb{P}(\mathfrak{q}_* \mathcal{L}')$ is a $\mathbb{P}^3$-bundle over $\mathbb{A}^1$, and  $\omega_{\mathfrak{q}}|_{C_{t}}$ gives the morphism $\mathfrak{g}\colon \mathcal{C} \rightarrow  \mathbb{P}(\mathfrak{q}_* \mathcal{L}')$. Then $\mathfrak{g}_t: =\mathfrak{g}|_{C_t}$ is the canonical embedding for $t\neq 0$, and the partial normalization $\mathfrak{g}_0 \colon C' \cup L \rightarrow C_{5} \cup L$ for $t=0$. Let  $E^{\bullet}$ be the two term complex $  \mathcal{O}_{\mathbb{P}(\mathfrak{q}_* \mathcal{L}')} \rightarrow \mathfrak{g}_{*}\mathcal{O}_{\mathcal{C}}$, and $ E_{t}$, the restriction of $ E^{\bullet}$ to the fiber over $t \in \mathbb{A}^{1}$. For $t \neq 0$, the map defining $E_t$ is surjective, and so $E \cong \mathcal{I}_{\mathfrak{g}(C_{t})}$, an object of  $\widetilde{\mathcal{M}_{\sigma_{-}}(v)} \cap \mathcal{M}'$. For $t=0$, the map fails to be surjective at the nodes, and so $\HH^{0}(E_{0})=\mathcal{I}_{C_{5} \cup L}$ and  $\HH^{1}(E_{0})=\mathcal{O}_{Z_{2}}$, and therefore $E_{0}$ is quasi-isomorphic to $E \in \mathcal{M}'$ itself.

  \end{proof}
  
  \begin{lem} \label{Uinint}
  We have $\mathcal{R}_1 \subset  \widetilde {\mathcal{M}_{\sigma_{-}}(v)} \cap \mathcal{M}'$.
  \end{lem}
\begin{proof}
We just constructed the set of ideal sheaves of canonical genus four curves in the intersection which is dense in $R_1$ (Corollary \ref{secondcor}). As the base of $\mathcal{R}_1$ is $U_1$, and the intersection is proper (and so every fiber will be contained in the intersection), then $\mathcal{R}_1$ will be contained in the intersection. 
\end{proof}

  \begin{rmk} \label{defrmk}
 We want to show that the intersection $\widetilde {\mathcal{M}_{\sigma_{-}}(v)} \cap \mathcal{M}'$ is exactly $\overline{\mathcal{R}_1}=\mathcal{R}$. 
 So far, we have shown that any object in $\mathcal{R}_1$ (corresponding to $U_1$ or $\mathrm{Ext}^1(A,B)=\mathbb{C}$) is contained in the intersection (as the specialization of an ideal sheaf of canonical genus 4 curves); i.e., $\mathcal{R}_1 \subset \widetilde {\mathcal{M}_{\sigma_{-}}(v)} \cap \mathcal{M}' \subset \mathcal{R}$ (Lemma \ref{Uinint}).  As $U_2 \subset \overline{U_1}$, the question is whether the objects in  $\mathcal{R} \setminus \mathcal{R}_1$ (corresponding to $U_2$ or $\mathrm{Ext}^1(A,B)=\mathbb{C}^2$)  
  are also in the intersection. We  show this by explicitly degenerating objects in $\mathcal{R}_1$ to get a 14-dimensional family of objects in $\overline{\mathcal{R}_1}$ lying over a given point of $U_2$. As the fiber in $\mathcal{R}$ over the given point, is 14-dimensional and irreducible, and as the intersection is closed, this shows that the whole fiber is contained in the closure. As $U_2$ is dense in the locus over which $\mathcal{R}$ is a cone bundle, this will prove that the whole cone bundle is contained in the intersection.

  We will simultaneously deform $L$ to   $\langle Z_2 \rangle$ and $C_5$ to $C_4 \cup   \langle Z_2 \rangle$ for a plane quartic $C_4$ containing $ Z_2$; the limit of the corresponding objects in  $\mathcal{R}_1$ is $\HH^0(E_{0})=\mathcal{I}_{C_4 \cup D}$ for a thickening $D$ of $\langle Z_2 \rangle$.

  \end{rmk}

  \begin{lem} \label{closure}
  We have $\overline{\mathcal{R}_1}=\mathcal{R}$, and thus $\mathcal{R} \subset \widetilde{\mathcal{M}_{\sigma_{-}}(v)} \cap \mathcal{M}' $.
  \end{lem}
  
  \begin{proof}
  
  To construct objects in the closure, we consider a family of objects in $\mathcal{R}_1$ parametrized by  $\mathbb{A}^1 \setminus \{0\}$, and then fill in the central fiber to an additional  object in $\widetilde{M_{\sigma_-}(v)}$. Let $C =\mathrm{S}pec \mathbb{C}[t]$, and $E \in \mathcal{R}_1$. As such objects are uniquely determined by their $\HH^0$, they are ideals $\mathcal I $ in $ \mathbb{C}[x, y, z, w][t, t^{-1}]$ and we want to find their flat limit as $t$ goes to zero (notice that to distinguish different limits, it is enough to distinguish them after restriction to  $\mathbb{A}^3$, so in what follows we consider the ideals in $\mathbb{A}^3$ instead of $\mathbb{P}^3$).

  By Remark \ref{defrmk}, we want to understand the infinitesimal direction of the thickening. There are two possibilities to deform $C_5 \cup L$ to $C_4 \cup D\colon$

   (1) Simultaneously, (in the plane $P'$ of $L$) deform $L$ directly to $\langle Z_2 \rangle$ with parameter $t$ ($L$ and all its deformations pass through $p=L \cap P$), and deform  $C_5$ to  $C_4 \cup \langle Z_2 \rangle$ with parameter $bt$ for some $b \in \mathbb{C}$.

   (2) Deform the plane $P'$ (which always will contain $L$ and its deformations $L_t$'s) with parameter c.

   For each class $E$ fitting in  $\mathcal{I}_{C_5 \cup L}  \hookrightarrow  E \twoheadrightarrow \mathcal{O}_{Z_{2}}[-1]$, we want to figure out how the deformations above (or deforming $b,c$) would affect it. To be more precise, let $P$ is given by $z=0$, and  $Z_2$ given by  $(0, 0, 0)$ and $(0, -1, 0)$, and so $\langle Z_2 \rangle$ is given by $x=z=0$. Also, let the deformations of $L$, i.e., $L_{t}$'s are given by $x=z-t(y-1)=0$, and we want to have $L_{0}$ identical to $\langle Z_2 \rangle$ in the end.
 Let us denote a branch of $C_4$ passing through $p_{1}$ by $L''$ and assume that $C_4$ is given by $f_4$. Furthermore, we set $z=bty^2(y+1)^2(y-1)=0$ to denote $(C_5)_{bt}$, i.e., deformations of $C_5$ to $C_4 \cup \langle Z_2 \rangle$. Therefore, (locally) we  have
 $$\mathcal{I}: =\mathcal{I}_{ L_t \cup (C_5)_{bt}}=(z,f_4.x+bty^2 (y+1)^2 (y-1)) \cap (x,z-t(y-1))$$$$=(zx,z^2 -tz(y-1),f_4.x^2 +btxy^2 (y+1)^2 (y-1),-tf_4x+btzy^2 (y+1)^2-bt^2 y^2 (y+1)^2 (y-1)).$$
Dividing the last term by $t$ and then letting $t$ go to zero, we have
    \begin{equation}\label{eqn: limit} 
    (zx, z^2, f_4x^2,-f_4x+bzy^2 (y+1)^2) \subset lim_{t \rightarrow 0}\mathcal{I}.
    \end{equation}

   The closure of this ideal in $\mathbb{P}^3$ represents  a degree four curve, $C_4$, plus a degree two infinitesimal thickening around $L'$. We also notice that the direction of the infinitesimal thickening is determined by the ratio of $f_4x$ and $bzy^2(y+1)^2$.

    To determine the thickening direction at each point $(x = 0, y, z=0)$ on $L'$, we can write the normal vectors to $L'$ as
\[ w_{\mathfrak{A}, \mathfrak{B}} = \mathfrak{A}(y) \frac{\partial}{\partial x} + \mathfrak{B}(y)\frac{\partial}{\partial z}.\]

The question is that for which $\mathfrak{A}, \mathfrak{B}$ we have
\[ w_{\mathfrak{A}, \mathfrak{B}}(\mathfrak{f}(0, y, 0)) = 0 \quad \text{for all $\mathfrak{f} \in lim_{t \rightarrow 0}\mathcal{I}$.} \]

As we mentioned above, we deform the plane $P'$ with the parameter $c$, by replacing $P'$ by another plane containing $L'$, i.e., replacing $x$ by $x-cz$ for some $c \in \mathbb{C}$. Now, we could ask how $\mathfrak{A}(y)$ and $\mathfrak{B}(y)$ depend on the choice of the parameters $b$ and $c$. 

We apply $w_{\mathfrak{A},\mathfrak{B}}$ to the last term of the ideal in \ref{eqn: limit}, replace $x$ by $x-cz$, and then plug in $(0, y, 0)$ to solve the following equation for $\mathfrak{A}(y)$ and $\mathfrak{B}(y)$
    $$0=w_{\mathfrak{A}, \mathfrak{B}} (-f_4(0,y,0)(x-cz)+bzy^2 (y+1)^2)$$$$=-\mathfrak{A}(y)f_4(0,y,0)+c\mathfrak{B}(y)f_4(0,y,0)+b\mathfrak{B}(y)y^2(y+1)^2.$$
   Therefore, we get
    $$\frac{\mathfrak{A}(y)}{\mathfrak{B}(y)}=b \frac{y^2(y+1)^2}{f_4(0,y,0)}+c. $$
     We notice that $lim_{y \rightarrow 0}(\frac{\mathfrak{A}(y)}{\mathfrak{B}(y)})=c$,  and then $b$ also can be recovered from  $\frac{\mathfrak{A}(y)}{\mathfrak{B}(y)}$ at any point $y\neq 0,-1$.

   Notice that we have $12$ dimension for the choice of $C_4$ in the plane $P$, and two more dimensions for the parameters $b$ and $c$, corresponding to the proportion of the deformations of $L$ and $C_5$, and the deformation of the plane $P'$ (containing L), respectively. Thus we have a $14$-dimensional locus, which is open in the irreducible 14-dimensional cone and so it has to be dense in each fiber; therefore, its closure has to be the entire cone. This means that the closure of $\mathcal{R}_1$ is the whole $\mathcal{R}$. By Lemma \ref{Uinint}, we have $\mathcal{R}_1 \subset \widetilde{\mathcal{M}_{\sigma_{-}}(v)} \cap \mathcal{M}'$, and as the intersection is closed, we have  $\overline{\mathcal{R}_1} \subset \widetilde{\mathcal{M}_{\sigma_{-}}(v)} \cap \mathcal{M}'$; but we just proved $\overline{\mathcal{R}_1}=\mathcal{R}$, so we have  $\mathcal{R} \subset \widetilde{\mathcal{M}_{\sigma_{-}}(v)} \cap \mathcal{M}'$.

 \end{proof}
 
\begin{rmk}We notice that, once we are given the infinitesimal thickening direction at each point on $L'$, we can have the speed of the deformation of the plane $P'$ as well as the proportion of the speeds of the deformations of $C_5$ and $L$. The normal direction points in the direction of $P'$ near $Z_2$  (i.e., the (deformations of) plane $P'$ could be recovered from the normal direction near $Z_2$), in the direction of $P$ near the other intersection points $C_4 \cap L'$, and varies in between depending on $b$ (see Figure \ref{tangent direction of the infinitesimal thickening}).

\begin{figure}[h]
    \includegraphics[width=12.7cm]{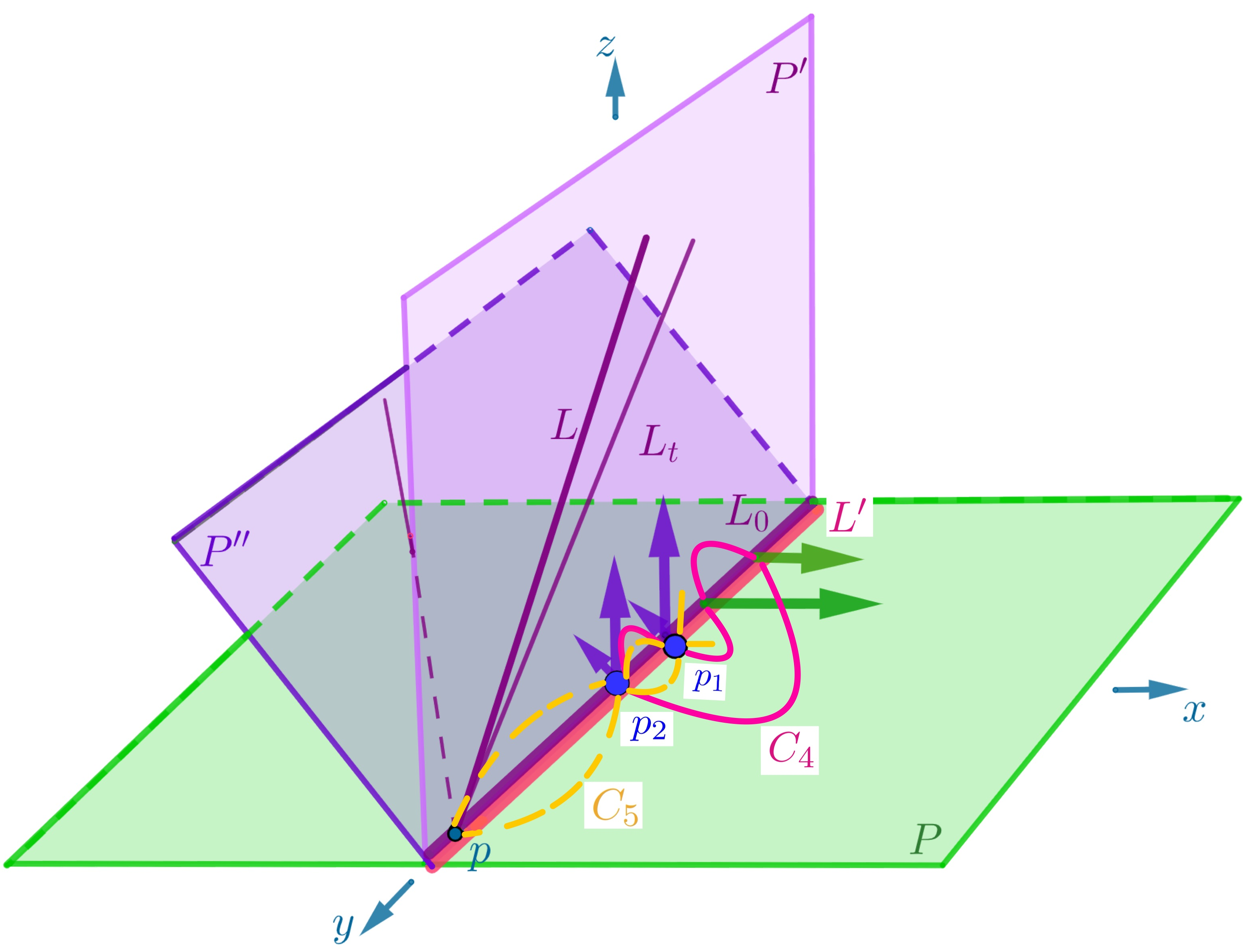}
\caption{Tangent directions of the infinitesimal thickening ($P''$ is a deformation of $P'$)}
\label{tangent direction of the infinitesimal thickening}
\end{figure}
      
\end{rmk}

Now we can show that $\mathcal{R}$ is exactly the intersection:

  \begin{cor} \label{singint}
Let $E \in \mathcal{M}'$. $E$ is in the intersection $\widetilde {\mathcal{M}_{\sigma_{-}}(v)} \cap \mathcal{M}'$ if and only if $ext^{1}(E,E) > 28$. In other words, the singularity locus of $\mathcal{M}'$ is exactly the intersection.
\end{cor}
\begin{proof}
This is deduced from Proposition \ref{P} and Lemma \ref{closure}.
\end{proof}

At this point, we can complete the proof of Theorem \ref{intN3}:

  \begin{thm} \label{intersec}
    
 We have
    $$\widetilde {\mathcal{M}_{\sigma_{-}}(v)} \cap \mathcal{M}'=\{\text{$E\colon \zeta$  associated to $E$ is  non-injective}\}.$$
    More precisely, we have $\widetilde{\mathcal{M}_{\sigma_{-}}(v)} \cap \mathcal{M}'=\mathcal{R}$ which is equal to the exceptional divisor of the contraction map $\psi$; i.e.,  the intersection  contains an open subset $\mathcal{R}_1$ such that $\psi|_{\mathcal{R}_1}$ is a $\mathbb{P}^{13}$-bundle over a 10-dimensional base. Over a 7-dimensional subset of the base, the fibers degenerate to a 14-dimensional cone with $\mathbb{P}^9$ as vertex over the 5-dimensional variety of rank one $2 \times 4$ matrices.

   \end{thm}

   \begin{proof}

 The first part is obtained from Lemma \ref{ext1E} and Corollary \ref{singint}, and noticing that $\widetilde{\mathcal{M}_{\sigma_{-}}(v)}   \rightarrow \mathcal{M}_{\sigma_{0}}(v) $ is the restriction of  $\mathcal{M}_{\sigma_{+}}(v) \dashrightarrow \mathcal{M}_{\sigma_{0}}(v)$, and the later contracts $ \mathcal{M}^{'}$. The second part is obtained from Lemma \ref{closure}.

   \end{proof}

\section{Birational description of the wall-crossing} \label{Birational morphism corresponding to the wall-crossing}

In this section, we will finally prove Theorem \ref{main Theorem}, which explains the birational relation between $\mathcal{M}_{\sigma_{-}}(v)$ and  $\widetilde{ \mathcal{M}_{\sigma_{-}}(v)}$. From the construction we know that $ \mathcal{M}_{\sigma_{-}}(v) \setminus \mathcal{U}_{-,+}$ is isomorphic to $\mathcal{M}_{\sigma_{+}}(v) \setminus \mathcal{M}'$, i.e., they both parametrize strictly stable objects with respect to stability conditions on the wall $\WW$, and stable  with respect to stability conditions on the both sides of $\WW$. Now the question is how we can describe the birational relationship between their closures. First, we have the following Lemma and Corollary:

\begin{lem} \label{goodmoduli}
There are maps $\mathcal{M}_{\sigma_{-}}(v) \rightarrow \mathcal{M}_{\sigma_{0}}(v)$ and $\widetilde{ \mathcal{M}_{\sigma_{-}}(v)} \rightarrow \mathcal{M}_{\sigma_{0}}(v)$  which contract the loci of $S$-equivalent objects, which are  a  $\mathbb{P}^{1}$-bundle and a 23-dimensional locus, respectively.

\end{lem}
\begin{proof} 
It follows from Lemma  \ref{U-+} and Theorem \ref{intersec} (notice that the good moduli space $\mathcal{M}_{\sigma_{0}}(v)$ does exist by Lemma \ref{Lem: wall-moduli}).
\end{proof}

\begin{cor} \label{Cor: smalldivisorial}
The map $\phi$ is a small contraction and $\psi$ is a divisorial contraction.
\end{cor}
\begin{proof}
Since the 8  dimensional exceptional locus of $\phi$ in $\mathcal{U}_{-,+}$ is a $\mathbb{P}^{1}$-bundle (Lemma \ref{U-+}),  there is a small contraction from $\mathcal{M}_{\sigma_{-}}(v)$ to the good moduli space for the wall $\mathcal{W}$, i.e., $\mathcal{M}_{\sigma_{0}}(v)$  (see Lemma \ref{Lem: wall-moduli} and Lemma \ref{goodmoduli}). On the other hand, by Theorem \ref{intersec}, the exceptional locus of $\psi$ in $\widetilde {\mathcal{M}_{\sigma_{-}}(v)}$ is of codimension one, so we have a divisorial contraction from $\widetilde {\mathcal{M}_{\sigma_{-}}(v)}$ to $\mathcal{M}_{\sigma_{0}}(v)$.

\end{proof}

At this point, we can complete the proof of Theorem  \ref{main Theorem}:

\begin{thm}  \label{finally!}
Fix $v=(1,0,-6,15)$. There is a wall-crossing with respect to Bridgeland stability conditions $\mathcal{M}_{\sigma_{-}}(v)\rightarrow \mathcal{M}_{\sigma_{+}}(v)$ between two moduli spaces separated by the wall $\mathcal{W}$ with the following properties:
 \begin{itemize}
     \item $\mathcal{M}_{\sigma_{-}}(v)$ is a smooth and irreducible variety,
     \item $\mathcal{M}_{\sigma_{+}}(v)= \widetilde {\mathcal{M}_{\sigma_{-}}(v)} \cup \mathcal{M}'$, where $ \widetilde {\mathcal{M}_{\sigma_{-}}(v)}$ is birational to $\mathcal{M}_{\sigma_{-}}(v)$ and $\mathcal{M}'$ is a new irreducible component,
     \item There is a diagram 
     \begin{center} \label{1}
    
\begin{tikzcd}[column sep=
scriptsize
]
 \mathcal{M}_{\sigma_{-}}(v) \arrow[dr, "\text{small contraction ($\phi$)}"{align=left,left=2mm,font=\scriptsize}] {}
&  &\widetilde {\mathcal{M}_{\sigma_{-}}(v)} \arrow[dl, "\text{divisorial contraction ($\psi$)}"{align=right,right=2mm,font=\scriptsize}] {} \\

& \mathcal{M}_{\sigma_{0}}(v)
\end{tikzcd}
\end{center}
     where both $\phi$ and $\psi$ have relative Picard rank 1. In particular, $ \widetilde {\mathcal{M}_{\sigma_{-}}(v)}$ is not $\mathbb{Q}$-factorial.
 \end{itemize}

\end{thm}
\begin{proof} The first three parts are implied by Proposition  \ref{N2}, Corollary \ref{Cor: M+}, and Corollary \ref{Cor: smalldivisorial}, respectively. The relative Picard rank of $\phi$ is one:  The only non-trivial fibers of $\phi$ are $\P^1$s, and these $\P^1$s are all numerically equivalent, as they occur in a connected family (over the locus where $\mathrm{Ext}^1(A, B)= \C^2$). In order to show that the relative Picard rank of $\psi$ is $1$, it is enough to show that the fibers have  $1-$dimensional $N_1$. For $\P^{13}$ this is obvious due to having a projective contraction. To show this for the 14-dimensional cone, we extend the method in {\cite[Example 2.8.]{FL}} from a cone with point vertex to the one with the $\mathbb{P}^9$ vertex:

 Let $X \subset \mathbb P(\mathbb C^{10} \oplus \mathbb C^8)$ be the projective cone with vertex $\mathbb P^9 = \mathbb P(\mathbb C^{10})$ over $Y$, the 4-dimensional variety corresponding to $2\times 4$-matrices of rank $\le 1$ in  $\mathbb P^7$. We want to describe the blow-up of $X$ at the vertex $\mathbb P^9$:  Let $Z$ be the $\mathbb P^{10}$-bundle over $Y$ with the fiber 
\[ \mathbb{P}(Z_y) = \mathbb P\Bigl(\mathbb{C}^{10} \oplus \text{one-dimensional subspace in $\mathbb C^8$ corresponding to y}\Bigr)\]
over $y \in Y$. Let us call this blow-up $\pi\colon Z \rightarrow X$. There is a natural map from $Z$ to $\mathbb P^{17}$ (with image contained in $X$) by forgetting the point $y$ and sending a point in the fiber $\mathbb P^{10} = \mathbb P(Z_y)$, corresponding to a one-dimensional subspace of $Z_y$, to the corresponding one-dimensional subspace of $\mathbb C^{18}$. On the other hand, there is a fibration $f: Z \rightarrow Y$. Thus we have 

 \begin{center}
     \begin{tikzcd}[column sep=
scriptsize
]
& Z \arrow[dl, "$f$" {align=left,left=0.0001mm,font=\scriptsize}]  \arrow[dr, "\pi" ] 
\\Y & & X. 
\end{tikzcd}
 \end{center}

Now, let $E:=Y \times \mathbb{P}^9$ be the exceptional locus of $\pi$, i.e., we have $\pi(E)=\mathbb{P}^9$. Following {\cite[Example 2.8.]{FL}}, as the relative dimension of $f$ is $10$, we have $N_1(Z)=E^{9}.f^{*}N_0(Y) \oplus E^{10}.f^{*}N_1(Y)$. Let $i$ be the embedding of $Y$ into $Y \times \P^9$ followed by  embedding in $Z$. Then as $\pi \circ i $ is constant, we will have $\pi_{*}(E^{10}.f^{*}N_1(Y))=0$. Now, taking $\pi_{*}$ of $N_1(Z)=E^{9}.f^{*}N_0(Y) \oplus E^{10}.f^{*}N_1(Y)$, we have $\pi_{*}N_1(Z)=\pi_{*}E^{9}.f^{*}N_0(Y)$, and we know that $N_1(X)$ is generated by $\pi_*N_1(Z)$. Notice that  $\pi$ is a contraction of a projective variety; hence, the equality indicates that $\pi_*N_1(Z)$ has rank at most one, and on the other hand, cannot have rank zero. Therefore, $\rk(N_1(X))=1$.

Finally, we need to show that $ \widetilde {\mathcal{M}_{\sigma_{-}}(v)}$ is not $\mathbb{Q}-$factorial. If it was $\mathbb{Q}-$factorial, then  $\mathcal{M}_{\sigma_{0}}(v)$ would also be $\mathbb{Q}-$factorial, since it is the image of $ \widetilde {\mathcal{M}_{\sigma_{-}}(v)}$ under a divisorial contraction of relative Picard rank 1 (see the proof of {\cite[Corollary 3.18]{KM}}). On the other hand,  $\mathcal{M}_{\sigma_{0}}(v)$ is the image of $\mathcal{M}_{\sigma_{-}}(v)$ (which is smooth and inparticular, $\mathbb{Q}-$factorial) under a small contraction, and hence cannot be $\mathbb{Q}-$factorial;   
this is a contradiction.

\end{proof}

We can see that it is not possible to explain the diagram in the statement of the above Theorem in $Mov(\mathcal{M}_{\sigma_{-}}(v))$ or in $Mov(\widetilde {\mathcal{M}_{\sigma_{-}}(v)})$. To explain the birational behaviour of the wall-crossing further, we have the following argument.

In what follows, by abuse of terminology, we use the terminology ``flip" for the small birational map with respect to the relevant small contractions. Let $ \mathcal{N}$ be the flip of $\mathcal{M}_{\sigma_{-}}(v)$ with respect to the small contraction $\phi$. Let $\mathbb{E}$ be the exceptional divisor in $\widetilde {\mathcal{M}_{\sigma_{-}}(v)}$, and consider $\mathbb{E}'=\psi (\mathbb{E})$; then $\psi$ is the blow-up of $\mathcal{M}_{\sigma_{0}}(v)$ along $\mathbb{E}'$. Now  define $\phi'\colon \NN \to \mathcal{M}_{\sigma_{0}}(v)$; then  we can blow up $\mathcal{M}_{\sigma_{-}}(v)$ and $ \mathcal{N}$ along $\phi^{-1}(\mathbb{E}')$ and $\phi'^{-1}(\mathbb{E}')$ and call the result $ \mathcal{N}''$ and $ \mathcal{N}'$, respectively. Notice that $\mathcal{N}'$ is the flip of $ \mathcal{N}''$ with respect to the small contraction $\gamma$. 

Now, under the map from $\Stab(\mathbb{P}^{3})$ to the movable cone $Mov(\mathcal{N}')= Mov(\mathcal{N}'')$, the wall $\mathcal{M}_{\sigma_{0}}(v)$ is not sent to a single wall in the movable cone, but it is sent to walls corresponding to the intersection of two, the divisorial and the flipping contractions. Furthermore, this map sends the two adjacent chambers not to the chambers but to a subset on the divisorial and flipping walls; in other words, crossing the wall $\mathcal{M}_{\sigma_{0}}(v)$ in the stability world is equivalent to remaining and walking on the broken line passing the point which is the image of the wall $\mathcal{M}_{\sigma_{0}}(v)$ in the birational world (see Figure \ref{birational model (or quasi-linearization) in the movable cone}). Notice that 
$\mathcal{N}$ is not induced by
  the associated wall inducing $\phi$, as by Theorem \ref{finally!},  one component of $\mathcal{M}_{\sigma_{+}}(v)$ is of a higher dimension and for the other component, we have the divisorial contraction $\psi$ instead of a small contraction.

The wall-crossing in the stability space is equivalent to the combination consists of the path starting from  $\mathcal{M}_{\sigma_{-}}(v)$ and heading to $\mathcal{M}_{\sigma_{0}}(v)$, followed by the path starting from $\mathcal{M}_{\sigma_{0}}(v)$ and heading to $\widetilde {\mathcal{M}_{\sigma_{-}}(v)}$ in the movable cone (see Figure \ref{birational model (or quasi-linearization) in the movable cone}). To summarize, we have the following diagram which explains the relation between $\mathcal{M}_{\sigma_{-}}(v)$ and $\widetilde {\mathcal{M}_{\sigma_{-}}(v)}$:

   \begin{center}
    
\begin{tikzcd}[column sep=
scriptsize
]
& & & & & & & & \textcolor{magenta}{\mathcal{N}''}  \arrow[dllllllll, rightarrow,"\eta "] \arrow[drr, dashrightarrow, "\beta'" ] \arrow[ddr, dashrightarrow, "\gamma" ]\\
\textcolor{green!75!black}{\mathcal{M}_{\sigma_{-}}(v)} \arrow["\phi",ddr ]   \arrow[drr, dashrightarrow, "\beta" ]{}
& & & & & & & & & & \textcolor{brown!70!black}{\mathcal{N}'} \arrow[dllllllll, rightarrow,"\theta "]   \arrow[dl, rightarrow,"\gamma'"]\\
& & \color{green!47!black}{\mathcal{N}} \arrow[dl,"\phi'"] & & & & & & & \textcolor{cyan!71!black}{\widetilde {\mathcal{M}_{\sigma_{-}}(v)}} \arrow[dllllllll, "\psi "]  \\
& \textcolor{violet!80!black}{\mathcal{M}_{\sigma_{0}}(v)}
\end{tikzcd}

\end{center}

    Existence of the small maps $\gamma$ and $\gamma'$ is implied by the universal property of blow-ups as the maps from $\mathcal{N}''$ and $ \mathcal{N}'$ to $\mathcal{M}_{\sigma_{0}}(v)$ must be factored via $\psi$  (notice that as $E'$ is closed in $\mathcal{M}_{\sigma_{0}}(v)$, the blow-up of its inverse image in $\mathcal{N}''$ and $ \mathcal{N}'$ must be Cartier). Pictorially, the movable cone of $\mathcal{N}'$ (or $\mathcal{N}''$) looks like the following:
 
  \begin{figure}[h]
    \includegraphics[width=12cm]{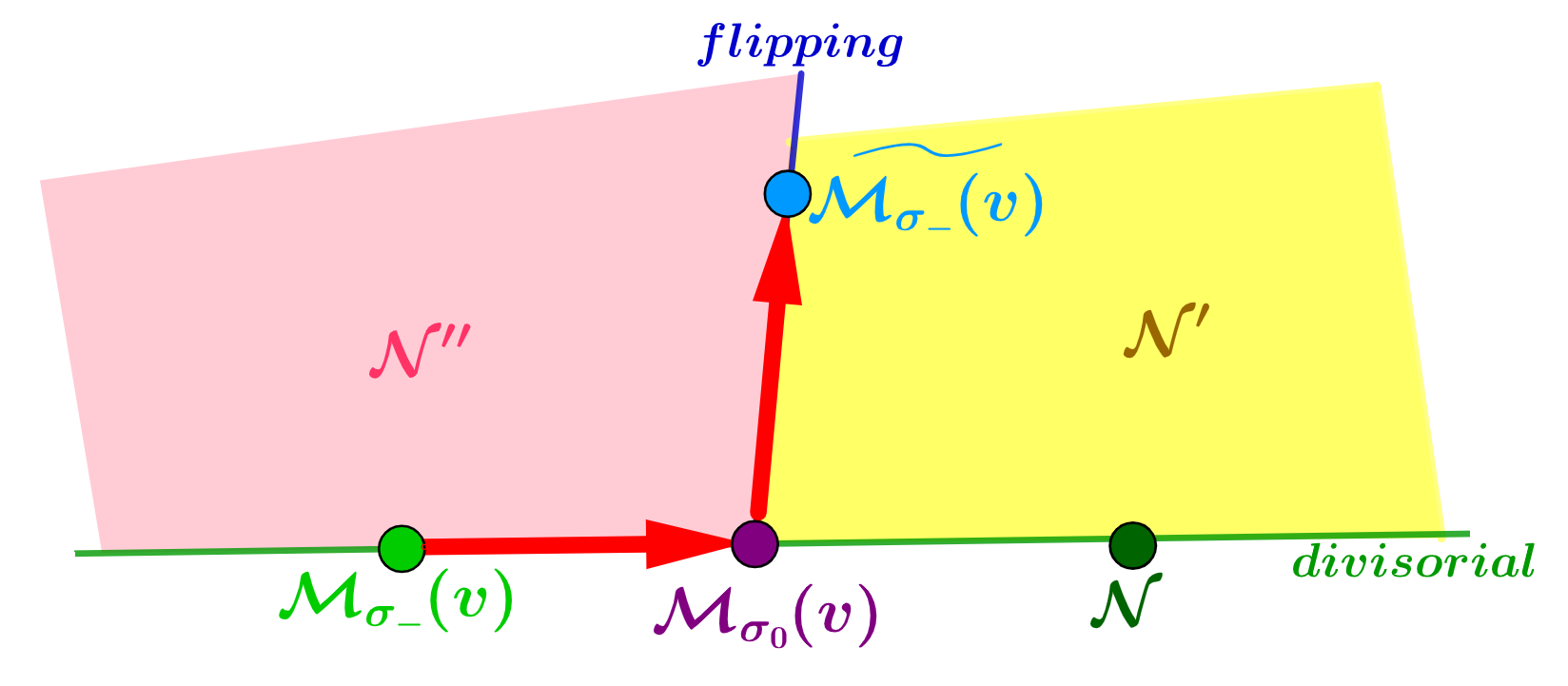}
\caption{Birational models in the movable cone of $\NN'$ (or $\NN''$)}
\label{birational model (or quasi-linearization) in the movable cone}
\end{figure}

\bibliographystyle{plain}
\bibliography{sample}

\end{document}